\documentclass[11pt]{amsart}

\usepackage{appendix} 

\usepackage{tabularx, environ}

\usepackage{amsmath, amsthm, amssymb, amsfonts} 

\allowdisplaybreaks 

\usepackage{eurosym} 

\usepackage{stmaryrd} 

\usepackage{bm} 

\usepackage[hidelinks]{hyperref} 
\usepackage{color} 
\usepackage[normalem]{ulem} 
\usepackage{comment} 

\usepackage{graphicx} 
\usepackage{wrapfig} 
\usepackage{caption} 
\usepackage{subcaption} 
\usepackage{multirow} 
\usepackage{array} 
\usepackage[section]{placeins} 

\usepackage{tikz} 
\usepackage{pgfplots} 
\pgfplotsset{compat=1.12} 

\newtheorem{thm}{Theorem}[section]
\newtheorem{cor}[thm]{Corollary}
\newtheorem{lem}[thm]{Lemma}
\newtheorem{prop}[thm]{Proposition}

\theoremstyle{definition}
\newtheorem{defn}[thm]{Definition}
\newtheorem{ass}[thm]{Assumption}

\theoremstyle{remark}
\newtheorem{rem}[thm]{Remark}

\DeclareFontFamily{U}{mathx}{\hyphenchar\font45}
\DeclareFontShape{U}{mathx}{m}{n}{<->mathx10}{}
\DeclareFontSubstitution{U}{mathx}{m}{n}
\DeclareSymbolFont{mathx}{U}{mathx}{m}{n}

\DeclareMathSymbol{\intop}  {\mathop}{mathx}{"B3}
\DeclareMathSymbol{\iintop} {\mathop}{mathx}{"B4}
\DeclareMathSymbol{\iiintop}{\mathop}{mathx}{"B5}
\DeclareMathSymbol{\ointop} {\mathop}{mathx}{"B6}
\DeclareMathSymbol{\oiintop}{\mathop}{mathx}{"B7}

\DeclareMathOperator{\sgn}{sgn} 

\DeclareFontFamily{U}{mathx}{\hyphenchar\font45}
\DeclareFontShape{U}{mathx}{m}{n}{<-> mathx10}{}
\DeclareSymbolFont{mathx}{U}{mathx}{m}{n}
\DeclareMathAccent{\widebar}{0}{mathx}{"73}

\newcommand{\R}{\mathbb{R}}

\newcommand{\N}{\mathbb{N}}

\newcommand{\E}{\mathbb{E}}

\newcommand{\ind}{\bm{1}}

\newcommand{\ud}{\mathrm{d}}

\newcommand{\uI}{\mathrm{I}}
\newcommand{\uII}{\mathrm{II}}
\newcommand{\uIII}{\mathrm{III}}
\newcommand{\uIV}{\mathrm{IV}}
\newcommand{\uV}{\mathrm{V}}

\newcommand{\be}{\begin{equation}}
\newcommand{\ee}{\end{equation}}
\newcommand{\ba}{\begin{aligned}}
\newcommand{\ea}{\end{aligned}}

\numberwithin{equation}{section}

\newcommand{\cF}{\mathcal{F}}

\newcommand{\cX}{\mathcal{X}}

\newcommand{\EE}{\mathbb{E}}
\newcommand{\FF}{\mathbb{F}}

\newcommand{\PP}{\mathbb{P}}

\newcommand{\tildeN}{\widetilde{N}}
\newcommand{\tildeX}{\widetilde{X}}

\newcommand{\hatN}{\widehat{N}}

\newcommand{\hatX}{\widehat{X}}
\newcommand{\barX}{\widebar{X}}

\newcommand{\red}{}

\newcommand{\dbra}[1]{[\kern-0.15em[ #1 ]\kern-0.15em]}
\newcommand{\dbraco}[1]{[\kern-0.15em[ #1 [\kern-0.15em[}
\newcommand{\dbraoc}[1]{]\kern-0.15em] #1 ]\kern-0.15em]}
\newcommand{\dbraoo}[1]{]\kern-0.15em] #1 [\kern-0.15em[}

\title[On non-negative solutions of stochastic Volterra equations with jumps]{On non-negative solutions of stochastic Volterra equations with jumps and non-Lipschitz coefficients}

\author[A. Alfonsi]{Aur\'elien Alfonsi}
\address{CERMICS, Ecole des Ponts, Marne-la-Vall\'ee, France. MathRisk, Inria, Paris,
  France.}
\email{aurelien.alfonsi@enpc.fr}

\author[G. Szulda]{Guillaume Szulda}
\address{CERMICS, Ecole des Ponts, Marne-la-Vall\'ee, France.}
\email{guillaume.szulda@enpc.fr}
\thanks{This work benefited from the support of the ``chaire Risques financiers'', Fondation du Risque.}

\date{\today}

\begin{document}

\begin{abstract}
	We consider one-dimensional stochastic Volterra equations with jumps for which we establish conditions upon the convolution kernel and coefficients for the strong existence and pathwise uniqueness of a non-negative c\`adl\`ag solution. By using the approach recently developed by~\cite{Alfonsi23}, we show the strong existence by using a nonnegative approximation of the equation whose convergence is proved via a variant of the Yamada--Watanabe approximation technique. We apply our results to L\'evy-driven stochastic Volterra equations. In particular, we are able to define a Volterra extension of the so-called \emph{alpha-stable Cox--Ingersoll--Ross process}, which is especially used for applications in Mathematical Finance.
\end{abstract}

\maketitle

\section{Introduction}\label{sec:sec1}
We consider one-dimensional stochastic Volterra equations with jumps and of convolution type with the following form:
\be\label{eq:SVE_intro}
\begin{aligned}
	X_t &= X_0 + \int_0^t {K(t-s)\,\mu(X_{s})\,\ud s} + \int_0^t {K(t-s)\,\sigma(X_{s})\,\ud B_s}\\
	&\quad+\int_0^t\int_{\,U} {K(t-s)\,\eta(X_{s-}, u)\,\tildeN(\ud s, \ud u)},
\end{aligned}
\ee
where $B$ is a Brownian motion and $N(\ud t, \ud u)$ is a Poisson random measure on $\R_+ \times U$ with compensator $\hatN(\ud t, \ud u) := \ud t\,\pi(\ud u)$ and compensated measure $\tildeN(\ud t, \ud u) := N(\ud t, \ud u) - \hatN(\ud t, \ud u)$, where $\pi$ is a $\sigma$-finite Borel measure on a complete separable metric space $U$. 

Continuous stochastic Volterra equations, i.e. when $\eta\equiv0$, were extensively studied by diverse authors from the 80s, see, among others, \cite{BM80a,BM80b,Protter85,PP90,CLP95,AN97,CD01,Wang08,Zhang10}. They have then attracted a renewed interest in Mathematical Finance, since the seminal paper \cite{GJR18} that advocates for rough volatility models. As a prominent example, a ``rough''
version of the well-known \emph{Cox--Ingersoll--Ross process} (CIR) was designed by  El Euch and Rosenbaum~\cite{EER19} with the following equation
\begin{equation*}
	X_t = X_0 + \int_0^t {K(t-s)\bigl(a - \kappa \,X_{s}\bigr)\ud s} + \sigma\int_0^t {K(t-s)\sqrt{X_{s}}\,\ud B_s},    
\end{equation*}   
where $a,\sigma \ge 0$, $\kappa \in \R$ and the fractional kernel $K(t)=\frac{t^{H-1/2}}{\Gamma(H+1/2)}$ with $H\in(0,1/2)$.
A Volterra extension of general affine diffusions was also proposed by  Abi Jaber et al.~\cite{AJLP19}. {\red Well-posedness of stochastic Volterra equations with non-Lipschitz coefficients and singular kernel is an active field of research. We mention here the pioneering work of Mytnik and Salisbury~\cite{MS15} and the recent work of Hamaguchi~\cite{Hamaguchi23}.}

The literature on stochastic Volterra equations with jumps is instead very recent, and shows a growing interest. Abi Jaber et al.~\cite{AJCLP21,AJ21} have elaborated a weak solution theory for stochastic Volterra equations driven by a semimartingale. Bondi et al.~\cite{BLP24} then derived, under an affine structure imposed upon the coefficients, a semi-explicit formula for the Fourier--Laplace transform of the solution. These results have been used in Mathematical Finance by Bondi et al.~\cite{BPS24} to develop a new stochastic volatility model.

The purpose of this paper is to establish conditions upon the kernel $K$ and coefficients $\mu,\,\sigma,\,\eta$ for the strong existence and pathwise uniqueness of a non-negative c\`adl\`ag solution of Equation~\eqref{eq:SVE_intro}. {\red When $K \in C^1$, $\mu$ is Lipschitz, $\sigma$ is $1/2$-H\"older and $\eta\equiv0$, \cite[Proposition B.3]{AJEE19a} obtained strong existence and pathwise uniqueness.} When $K\equiv 1$ and $\eta$ not identically zero, Fu and Li~\cite{FL10} and  Li and Mytnik~\cite{LM11} exploited the Markov property of solutions to ensure the non-negativity and thus proved, under regularity conditions upon the coefficients of Yamada--Watanabe type (see \cite{YW71}), the strong existence and pathwise uniqueness of a non-negative c\`adl\`ag solution. They notably applied their results to reconstruct \emph{Continuous-state Branching processes with Immigration} (CBI), initially introduced by \cite{KW71}, which form an important class of non-negative Markov processes with non-negative jumps and also include the CIR process as a special case. An important example of CBI process exhibiting jumps is the \emph{alpha-stable Cox--Ingersoll--Ross process}, which consists in extending the CIR process by adding jumps of alpha-stable type as follows:
\be\label{eq:SDE_Levy_intro}
X_t = X_0 + \int_0^t {\bigl(a - \kappa\,X_{s}\bigr)\ud s} + \sigma\int_0^t {\sqrt{X_{s}}\,\ud B_s} + \eta\int_0^t {\sqrt[\alpha]{X_{s-}}\,\ud L_s}, 
\ee
where $L$ is a spectrally positive compensated $\alpha$-stable L\'evy process with $\alpha\in(1,2)$. By Fu and Li~\cite[Corollary 6.3]{FL10}, there exists a pathwise unique non-negative c\`adl\`ag strong solution of Equation~\eqref{eq:SDE_Levy_intro}. This process and related ones have been used for practical applications, notably in Mathematical Finance, see, e.g., Jiao et al.~\cite{JMS17,JMSS19,JMSZ21} or Fontana et al.~\cite{FGS21rate}.

The main result of the paper (Theorem \ref{thm:main_result} thereafter) reads as follows. Suppose that the kernel $K:\R_+\to\R_+$ is non-negative, non-increasing, twice continuously differentiable and \emph{preserves non-negativity} in the sense of Alfonsi~\cite[Definition 2.1]{Alfonsi23} such that $0 < K(0) < +\infty$. Suppose further that the coefficients $\mu,\,\sigma,\,\eta$ satisfy classical regularity conditions that are essentially those of Li and Mytnik~\cite{LM11}. Then, there exists a pathwise unique non-negative c\`adl\`ag strong solution of Equation~\eqref{eq:SVE_intro}. Let us note that the family of \emph{completely monotone} kernels, which is mainly used in practical applications, satisfy the non-negativity preserving property.   While the previous works mentioned on stochastic Volterra equations with jumps~\cite{AJCLP21,BLP24} deal with weak solution and square integrable jumps (i.e. $\int_{\,U} {\eta(x, u)^2\,\pi(\ud u)}<+\infty$, for all $x\in\R$), we work here with strong solutions and jumps satisfying $\int_{\,U} {|\eta(x, u)|\wedge  \eta(x, u)^2 \,\pi(\ud u)}<+\infty$, for all $x\in\R$, as in~\cite{LM11}. This latter point is crucial to obtain a generalization of~\eqref{eq:SDE_Levy_intro} to Volterra equations. Thus, in the present paper, we are able (see Corollary \ref{cor:Levy_result}) to prove the strong existence and pathwise uniqueness of a non-negative c\`adl\`ag solution of the following equation
\begin{equation*}
	X_t = X_0 + \int_0^t {K(t-s)\bigl(a - \kappa\,X_{s}\bigr)\ud s} + \sigma\int_0^t {K(t-s)\sqrt{X_{s}}\,\ud B_s} + \eta\int_0^t {K(t-s)\sqrt[\alpha]{X_{s-}}\,\ud L_s},
\end{equation*}
when $K$ is completely monotone with $0 < K(0) < +\infty$. This process can be seen as a Volterra alpha-stable Cox--Ingersoll--Ross process. 

The strategy we adopt for the strong existence is based upon an approximation of Equation~\eqref{eq:SVE_intro} initially introduced in~\cite{Alfonsi23}. It consists in splitting the convolution with kernel $K$ from the integration of the stochastic differential equation. The crucial property of this approximating process is that it stays nonnegative, relying on the kernel properties and on the results of~\cite{LM11}. The convergence of this approximation is then proved by using a variant of the Yamada--Watanabe functions (see Proposition \ref{prop:barX}). The latter was used, e.g., by \cite{Yamada78,Alfonsi05,GR11,LT19a,LT19b} in the context of (standard) stochastic differential equations. It has also been used very lately by \cite{PS23}, where the authors study continuous stochastic Volterra equations  which are not necessarily of convolution type and have a H\"older continuous diffusion coefficient. With respect to this work, we have two new difficulties in our framework. The first one is to handle the jumps when using Yamada--Watanabe functions: this difficulty is overcome by the technical Lemma~\ref{lem:lem_varphi_2} that allow to compare precisely enough the jumps of two approximating processes.  The second difficulty is that the processes that we consider only have a finite first moment, because of our assumption on the jumps. We can however take advantage of the non-negativity of our approximating processes to get a uniform upper bound of the first moment (Proposition~\ref{prop:first_moment_estimates}).

The paper is structured as follows. Section~\ref{sec:sec2} introduces the framework and the main result of the paper. We prove the pathwise uniqueness and strong existence for Equation~\eqref{eq:SVE_intro} respectively in Sections \ref{sec:sec3} and \ref{sec:sec4}.  Section \ref{sec:sec6} applies our main result to L\'evy-driven stochastic Volterra equations. We provide some auxiliary results in Appendix~\ref{app:sec5} that are used throughout the paper. Last, Appendix \ref{sec:appendix} contains a description of the variant of the Yamada--Watanabe approximation that we use in Section \ref{sec:sec4} as well as key technical lemmas.

\section{Assumptions and main result}\label{sec:sec2}
Let $(\Omega,\cF,\FF:=(\cF_t)_{t\geq0},\PP)$ be a given filtered probability space satisfying the usual conditions and supporting the following independent random elements:
\begin{itemize}
	\item an $\FF$-Brownian motion $B=(B_t)_{t\geq0}$;
	\item an $\FF$-Poisson point random measure $N(\ud t, \ud u)$ on $[0, +\infty) \times U$ with compensator $\hatN(\ud t, \ud u) := \ud t\,\pi(\ud u)$, where $U$ is a complete separable metric space on which $\pi$ is a $\sigma$-finite Borel measure. We denote by $\tildeN(\ud t, \ud u) := N(\ud t, \ud u)-\hatN(\ud t, \ud u)$ its compensated measure.
\end{itemize}
Let us also consider the following ingredients:
\begin{itemize}
	\item $\eta : \R \times U \to \R$ is a Borel function;
	\item $\mu, \sigma : \R \to \R$ are continuous functions;
	\item $K : \R_+ \to \R_+$  is a non-negative continuous function.
\end{itemize}
Note that this implies in particular that $K(0) < +\infty$.

For $X_0\in\R$, we concentrate upon the following one-dimensional stochastic Volterra equation with jumps and of convolution type:
\be\label{eq:SVE}
\begin{aligned}
	X_t &= X_0 + \int_0^t {K(t-s)\,\mu(X_{s})\,\ud s} + \int_0^t {K(t-s)\,\sigma(X_{s})\,\ud B_s}\\
	&\quad+\int_0^t\int_{\,U} {K(t-s)\,\eta(X_{s-}, u)\,\tildeN (\ud s, \ud u)}.
\end{aligned}
\ee
By a \emph{c\`adl\`ag solution} of Equation \eqref{eq:SVE}, we mean an almost surely c\`adl\`ag and $\FF$-adapted stochastic process $X=(X_t)_{t \geq 0}$ that satisfies Equation~\eqref{eq:SVE} almost surely for all $t\geq0$. In particular, the integrals appearing in the right hand side of Equation~\eqref{eq:SVE} are assumed to be well defined. This notion of solution corresponds to the usual notion of strong solution. 
We also say that a c\`adl\`ag solution of Equation~\eqref{eq:SVE} is \emph{non-negative} if we have $\PP(X_t\geq0, \forall t\geq 0)=1$. For an $\FF$-stopping time $\tau:\Omega \to [0,+\infty]$, we will say that a c\`adl\`ag and $\FF$-adapted stochastic process $X=(X_t)_{t \geq 0}$ is a \emph{c\`adl\`ag solution up to $\tau$} if the process
$$\left(X_0 + \int_0^{t\wedge \tau} {K(t-s)\,\mu(X_{s})\,\ud s} + \int_0^{t\wedge \tau} {K(t-s)\,\sigma(X_{s})\,\ud B_s}+\int_0^{t\wedge \tau} \int_{\,U} {K(t-s)\,\eta(X_{s-}, u)\,\tildeN (\ud s, \ud u)} \right)_{t\ge 0} $$
is well defined and is equal to $X_t$ for $t\in[0,\tau)$. 
As far as the well-posedness of Equation~\eqref{eq:SVE} is concerned, we then impose the following global condition on the coefficients $\mu, \sigma$ and $\eta$ of \eqref{eq:SVE}:
\begin{ass}\label{ass:conditions1}
	Suppose that there exists a constant $L > 0$ such that 
	\begin{equation*}
		\bigl|\mu(x)\bigr| + \sigma(x)^2 + \int_{\,U} {\!\bigl(\bigl|\eta(x, u)\bigr|\wedge\eta(x, u)^2\bigr)\pi(\ud u)} \leq L\bigl(1 + |x|\bigr), \qquad \text{for all } x\in\R.
	\end{equation*}
\end{ass}
Thus, under Assumption \ref{ass:conditions1}, Lemma \ref{lem:wellposedness} guarantees that the stochastic integrals appearing on the right-hand side of Equation~\eqref{eq:SVE}, taken with respect to any c\`adl\`ag $\FF$-adapted stochastic process $X=(X_t)_{t \geq 0}$, are well defined for all $t \geq 0$. In addition, we say that \emph{pathwise uniqueness} holds for Equation~\eqref{eq:SVE} if, for any two c\`adl\`ag solutions $X$ and $Y$ of Equation~\eqref{eq:SVE} in the above sense with $X_0 = Y_0$, we have $\PP(X_t = Y_t, \forall t \geq 0)=1$. This is the classical concept of pathwise uniqueness that can also be found, e.g., in \cite[Definition IV.1.5]{IW89}. By using Assumption \ref{ass:conditions1}, we have the following a priori estimates. 
\begin{lem}\label{lem:moments1}
	Let Assumption \ref{ass:conditions1} hold and $K:\R_+\to \R_+$ be a non-negative continuous function. Let $X=(X_t)_{t \geq 0}$ be a c\`adl\`ag solution of Equation~\eqref{eq:SVE}. Then, for every $T>0$, there exists a constant\footnote{In the sequel, the dependence of constants with respect to the parameters will be indicated in subscript.} $C_{T,L,K,X_0}\in \R_+$ depending on $T>0$, the constant $L$ of Assumption \ref{ass:conditions1}, the kernel~$K$ and $X_0$  such that
	\begin{equation*}
		\sup_{t\in[0,T]}\EE[|X_t|] \leq C_{T,L,K,X_0}.
	\end{equation*}
\end{lem}
\begin{proof}
	Since $X$ is c\`adl\`ag and $\FF$-adapted, $\tau_m := \inf\{t\geq 0 : |X_t| \geq m\}$ (with the usual convention $\inf \emptyset=+\infty$) is an $\FF$-stopping time for every $m\geq1$ (see, e.g., \cite[Example I.5.1]{IW89}). Besides, $\tau_m\to \infty$ almost surely as $m\to \infty$ since the paths of $X$ are c\`adl\`ag and thus locally bounded. We first write from Equation~\eqref{eq:SVE}
	\begin{align*}
		|\ind_{t<\tau_m} X_{t}|\le & \Bigl|X_0 +  \int_0^{t\wedge \tau_m} {K(t-s)\, \mu(X_{s})\,\ud s} + \int_0^{t\wedge \tau_m} {K(t-s)\,\sigma(X_{s})\,\ud B_s}\\
		&+\int_0^{t\wedge \tau_m}\int_{\,U} {K(t-s)\,\eta(X_{s-}, u)\,\tildeN (\ud s, \ud u)} \Bigr|. 
	\end{align*}
	This inequality is clear on $\{t\ge \tau_m \}$ and is an equality on $\{t<\tau_m\}$.  Then, by applying Proposition~\ref{prop:prem_first_moment_estimate} with $p=0$, $\tau=t\wedge\tau_m$, $q=t$ and $H(t,s)=K(t-s)$, we get 
	$$\E[\ind_{t<\tau_m} |X_{t}|]\le |X_0|+  C_L \left(\max_{[0,t]} K \right) \left(1+ 2t +2\int_0^t \E[\ind_{s<\tau_m}|X_{s}|] \ud s \right).$$
	Let $T>0$ and $C_{T,L,K}= 2C_L \left(\max_{[0,T]} K \right) (1+ T)$. We have
	\begin{equation*}
		\EE\bigl[\bigl| \ind_{t<\tau_m} X_{t} \bigr|\bigr] \leq |X_0|+C_{T,L,K} + C_{T,L,K}\int_0^t {\EE\bigl[\bigl|\ind_{s<\tau_m} X_{s}\bigr|\bigr]\ud s},
	\end{equation*}
	We conclude by using Gronwall's Lemma and the monotone convergence theorem as $m\to \infty$. 
\end{proof}

\begin{rem} 
	Unlike \cite{AJCLP21,PS23}, we cannot have higher moments than the first-order moment in  Lemma~\ref{lem:moments1}. This is because $\int_{\{u \in U : |\eta(x, u)|\geq1\}} {\bigl|\eta(x, u)\bigr|\,\pi(\ud u)}<+\infty$, for all $x\in\R$, as imposed by Assumption \ref{ass:conditions1}. This constraint is necessary for the Poisson random measure $N$ to represent the jumps of an alpha-stable L\'evy process (see Section \ref{sec:sec6}). {\red Let us also note that we need to assume $\sigma(x)^2 \leq L(1+|x|)$, for all $x\in\R$, in order to get first-order moments. Working with $L^1$ norms will reveal to be crucial later when dealing with non-negative processes $X$, for which we have $\EE[|X_t|] = \EE[X_t]$. This was notably used by \cite[Proposition 2.3]{FL10}.}
\end{rem}
Let us now formulate our local regularity conditions on the coefficients $\mu, \sigma$ and $\eta$ of Equation~\eqref{eq:SVE}. 
\begin{ass}\label{ass:local_conditions}
	Suppose that
	\begin{enumerate}
		\item[(i)] for every $m\geq1$, there exists a constant $L_m^{\prime} > 0$ such that 
		\begin{equation*}
			\bigl|\mu(x) - \mu(y)\bigr| + \bigl|\sigma(x) - \sigma(y)\bigr|^2 \leq L_m^{\prime}\,\bigl|x - y\bigr|, \qquad \text{for all } (x,y)\in[-m,m]^2;
		\end{equation*}
		\item[(ii)]   the function $x\mapsto\eta(x,u)$ is non-decreasing for every $u \in U$ and, for every $m\geq1$, there exists a non-negative Borel function $f_m:U\to\R_+$ such that 
		\begin{equation*}
			\bigl|\eta(x, u) - \eta(y, u)\bigr| \leq \bigl|x - y\bigr|^{1/2}f_m(u), \qquad \text{for all } (x,y,u) \in [-m,m]^2 \times U,
		\end{equation*}
		where $f_m$ satisfies $\int_{\,U} {(f_m(u) \wedge f_m(u)^2)\,\pi(\ud u)} < +\infty$.
	\end{enumerate}
\end{ass}
\begin{ass}\label{ass:nonneg}
	The coefficients satisfy $\sigma(0)=0$, $\mu(0)\geq0$ and $\eta(0,u)=0$ for all $u\in U$.
\end{ass}
\noindent The conditions stated in Assumptions~\ref{ass:conditions1}, \ref{ass:local_conditions} and~\ref{ass:nonneg} are rather close to those of Li and Mytnik~\cite[Equations (2b-f)]{LM11}. As illustrated by Lemma~\ref{lem:moments1}, Assumption~\ref{ass:conditions1} gives us bounded first order moments, while  Assumption~\ref{ass:nonneg} is crucial for the nonnegativity of the solution. The Assumption~\ref{ass:local_conditions} on the local regularity will be used for both strong existence and pathwise uniqueness of the solution. Note that our Assumption~\ref{ass:local_conditions} is a bit stronger than~\cite[Equations (2b-c)]{LM11} that corresponds to Assumption~\ref{ass:YW_conditions} below. We will be able to show pathwise uniqueness under Assumption~\ref{ass:YW_conditions}, but we need Assumption~\ref{ass:local_conditions} for the strong existence. Roughly speaking, this is due to our approach using an approximating sequence of processes.  We use a doubly-indexed variant of the Yamada--Watanabe functions for their convergence (see Appendix~\ref{sec:appendix}) that use the particular behaviour of the square-root function.

We now turn to the assumptions on the kernel function~$K$. The main difficulty is to have sufficient conditions on~$K$ that guarantee the nonnegativity of the solution~$X$ of~\eqref{eq:SVE}, which motivates the following definition. 
\begin{defn}\cite[Definition 2.1]{Alfonsi23}\label{def:non_negativity} 
	Let $K : \R_+ \to \R_+$ be a non-negative function such that $0 < K(0) < +\infty$. $K$ is said to \emph{preserve non-negativity} if, for any $M\in \N^*$, any $x_1,\dots,x_M \in \R$ and any $0 \leq t_1 < \dots < t_M$ such that 
	\begin{equation*}
		\sum_{j=1}^m x_j\,K(t_m - t_j) \geq 0, \qquad \text{for every } m \in \{1, \dots, M\},
	\end{equation*}
	it holds that 
	\begin{equation*}
		\sum_{m=1}^M \ind_{\{t_m \leq t\}} x_m\,K(t-t_m)\, \geq 0, \qquad \text{for all } t\geq0.
	\end{equation*}
\end{defn}
We are now in position to state our main result, whose proof is postponed at the end of Section~\ref{sec:sec4}.
\begin{thm}\label{thm:main_result}
	Suppose that $X_0\geq0$, Assumptions \ref{ass:conditions1}, \ref{ass:local_conditions}  and~\ref{ass:nonneg} hold true, $K \in C^2(\R_+,\R_+)$ is non-increasing, preserves non-negativity and such that $K(0)>0$. Then, there exists a pathwise unique non-negative c\`adl\`ag solution $X=(X_t)_{t\geq0}$ of Equation \eqref{eq:SVE}.
\end{thm}
It is shown in~\cite[Theorem 2.3]{Alfonsi23} that completely monotone kernels preserve nonnegativity.  Recall that a function $K:\R_+\to \R_+$ such that $ K(0)>0$ is said to be \emph{completely monotone} if $K \in C^{\infty}(\R_+,\R_+)$ such that $(-1)^n\,K^{(n)}\geq0$ for every $n\geq0$. By Bernstein's theorem, this is equivalent to the the existence of a finite (non trivial) Borel measure $\theta$ on $\R_+$ such that 
\begin{equation*}
	K(t) = \int_0^{+\infty} {e^{-\lambda t}\,\theta(\ud \lambda)}, \qquad \text{for all } t\geq0.
\end{equation*}
For example, if we set $\theta(\ud \lambda) = \sum_{i=1}^n w_i \delta_{\lambda_i}(\ud \lambda)$ with $n\geq1$ and $0 <\lambda_1 < \lambda_2 < \cdots < \lambda_n$ and $w_1,\dots,w_n>0$, then $K(t) = \sum_{i=1}^n w_i e^{-\lambda_i t}$, for all $t\geq0$, defines a completely monotone function. 
For $X_0\geq0$, consider the stochastic Volterra equation
\be\label{eq:sum_exponentials}
X_t = X_0 + \sum_{i=1}^n\int_0^t w_i {e^{-\lambda_i(t-s)}\biggl(\mu(X_{s})\,\ud s + \sigma(X_{s})\,\ud B_s + \int_{\,U} {\eta(X_{s-}, u)\,\tildeN (\ud s, \ud u)}\biggr)}.
\ee
By Theorem~\ref{thm:main_result}, there exists a pathwise unique non-negative c\`adl\`ag solution $X=(X_t)_{t\geq0}$ of Equation~\eqref{eq:sum_exponentials}. If we define for $i\in \{1,\dots,n\}$ the processes 
$$X^i_t=\int_0^t {e^{-\lambda_i(t-s)}\biggl(\mu(X_{s})\,\ud s + \sigma(X_{s})\,\ud B_s + \int_{\,U} {\eta(X_{s-}, u)\,\tildeN (\ud s, \ud u)}\biggr)},$$
we get $X_t=X_0+\sum_{i=1}^n w_i X^i_t$ with
$$dX^i_t=-\lambda X^i_t dt +\mu(X_{t})\,\ud t + \sigma(X_{t})\,\ud B_t + \int_{\,U} {\eta(X_{t-}, u)\,\tildeN (\ud t, \ud u)}.$$
Thus, $(X^1_t,\dots,X^n_t)$ solves a classical SDE with jumps. But even in this case, it may be not obvious that this SDE admits a solution and that $X_t$ remains non-negative. This is the contribution of Theorem~\ref{thm:main_result}. Note that combining this discussion with the results of Section \ref{sec:sec6} (e.g. Corollary \ref{cor:Levy_result}), we can define a multifactor alpha-CIR process.

{\red\begin{rem}
  It would be interesting to deal with completely monotone kernels exploding in zero (e.g. the fractional kernel). This however raises important technical difficulties. In the case without jumps, \cite{MS15} obtain pathwise uniqueness and strong existence for the fractional kernel but with a diffusion coefficient $\sigma(x)=c |x|^\gamma$ with $\gamma>1/2$ while \cite{AJEE19a} consider $\gamma=1/2$ but require as in our work a non-exploding kernel. In the case with jumps, \cite{AJCLP21,BLP24}  obtain weak solutions with a kernel possibly exploding in zero, but assume square integrable jumps (i.e. $\int_{\,U} {\eta(x, u)^2\,\pi(\ud u)}<+\infty$, for all $x\in\R$). Here, we get strong solutions with jumps satisfying $\int_{\,U} {|\eta(x, u)|\wedge  \eta(x, u)^2 \,\pi(\ud u)}<+\infty$, for all $x\in\R$, but we have to assume $K(0)<+\infty$.
\end{rem}}

\section{Pathwise uniqueness}\label{sec:sec3}
In this section, we investigate the pathwise uniqueness of c\`adl\`ag solutions of Equation~\eqref{eq:SVE}. We are able to prove it under slightly weaker regularity conditions on the coefficients $\mu, \sigma$ and $\eta$ than those of Assumption \ref{ass:local_conditions}.
\begin{ass}\label{ass:YW_conditions}
	Suppose that
	\begin{enumerate}
		\item[(i)] for every $m\geq1$, there exists a non-decreasing and concave function $r_m:\R_+\to\R_+$ where $r_m(0)=0$, $r_m(x) > 0$ for $x>0$ and $\int_0^{\varepsilon} {r_m(x)^{-1}\ud x} = +\infty$ for all $\varepsilon > 0$, such that 
		\begin{equation*}
			\bigl|\mu(x) - \mu(y)\bigr| \leq r_m\bigl(|x-y|\bigr), \qquad \text{for all } (x,y)\in[-m,m]^2;
		\end{equation*}
		\item[(ii)]  the function $x\mapsto\eta(x,u)$ is non-decreasing for every $u \in U$ and, for every $m\geq1$, there exist a non-negative Borel function $f_m:U\to\R_+$ and a non-decreasing function $\rho_m:\R_+\to\R_+$ with $\rho_m(0)=0$, $\rho_m(x) > 0$ for $x>0$ and $\int_0^{\varepsilon} {\rho_m(x)^{-2}\ud x} = +\infty$ for all $\varepsilon > 0$, such that
		\begin{equation*}
			\bigl|\sigma(x) - \sigma(y)\bigr| \leq \rho_m\bigl(|x-y|\bigr) \quad \text{and} \quad \bigl|\eta(x,u) - \eta(y,u)\bigr| \leq \rho_m\bigl(|x-y|\bigr)\,f_m(u), 
		\end{equation*}
		for all $(x,y,u)\in[-m,m]^2 \times U$, where $f_m$ satisfies $\int_{\,U} {(f_m(u) \wedge f_m(u)^2)\,\pi(\ud u)} < +\infty$.
	\end{enumerate}
\end{ass} 
\noindent Assumption~\ref{ass:YW_conditions} essentially corresponds to the assumptions made by Li and Mytnik~\cite[Equations (2b-c)]{LM11}. Moreover, as for~\cite[Propositions 3.1 and 3.3]{LM11}, our result extends to c\`adl\`ag solutions of Equation \eqref{eq:SVE} which are not necessarily non-negative.
\begin{thm}\label{thm:pathwise_uniqueness}
	Suppose that $K \in C^2(\R_+,\R_+)$ with $K(0)>0$, Assumptions \ref{ass:conditions1} and \ref{ass:YW_conditions} hold true. Then, pathwise uniqueness holds for Equation~\eqref{eq:SVE}. More precisely, if $\tau:\Omega \to [0,+\infty]$ is an $\FF$-stopping time and $X,Y$ are two c\`adl\`ag solutions of Equation~\eqref{eq:SVE} up to~$\tau$, then it holds that $\PP(X_t=Y_t, \forall t \in[0,\tau))=1$. 
\end{thm}
\begin{proof} 
	Let $X$ and $Y$ be two c\`adl\`ag solutions of Equation~\eqref{eq:SVE} up to~$\tau$ with $X_0 = Y_0$ and define, for every $m\geq1$, $\tau_m := \inf\bigl\{t\geq0 : |X_t| \geq m \text{ or } |Y_t|\geq m\bigr\}$. Since $X=(X_t)_{t\geq0}$ and $Y=(Y_t)_{t\geq0}$ are c\`adl\`ag and $\FF$-adapted by definition, $\tau_m$ is an $\FF$-stopping time for every $m\geq1$ and we have $\PP(\tau_m\to+\infty,\,m\to+\infty)=1$. We define then the process $Z$ by 
	$$ Z_t=\int_0^t K(t-s) \,\ud H_s, $$
	where we define the process $H_t=(H_t)_{t\geq0}$ as
	\begin{equation*}
		H_t := \int_0^{t} \ind_{s<\tau}{\biggl(\bigl(\mu(X_{s}) - \mu(Y_{s})\bigr)\ud s + \bigl(\sigma(X_{s})-\sigma(Y_{s})\bigr)\ud B_s + \int_{\,U} {\bigl(\eta(X_{s-}, u)-\eta(Y_{s-}, u)\bigr)\tildeN (\ud s, \ud u)}\biggr)}.
	\end{equation*} 
	By construction, we have $Z_t=X_t-Y_t$ for $t\in[0,\tau)$. Using $K \in C^2(\R_+,\R_+)$ and Proposition~\ref{prop:semi_martingale}, we also have
	\begin{equation*}
		Z_{t} = K(0)\,H_{t  } + K'(0) \int_0^{t} { H_{s} \ud s}+\int_0^{t} {\biggl(\int_0^s {K''(s-r) H_{r}} \ud r \biggr)\ud s},
	\end{equation*}
	so that the process $Z$ is c\`adl\`ag. 
	In addition, the process $(Z_{t\wedge \tau_m})_{t\ge 0}$ is an $\FF$-semimartingale. 
	Under Assumption \ref{ass:YW_conditions}-(ii), we consider a sequence $(a_k)_{k\in\N}\in(0,1]^{\N}$ such that 
	\begin{equation*}
		a_0=1, \qquad a_{k-1} > a_k, \qquad a_k \underset{k\to+\infty}{\longrightarrow} 0, \quad \text{and} \quad \int_{a_k}^{a_{k-1}} {\frac{\ud x}{\rho_m(x)^2}} = k,
	\end{equation*}
	for every $k\geq1$. Following the argument used in the proof of \cite[Theorem 1]{YW71}, we can construct, for every $m\geq1$, a sequence of smooth functions $(\varphi_k)_{k\in\N} \in C^2(\R,\R)^{\N}$ such that $\varphi_k(0)=0$, $\varphi_k(-x)=\varphi_k(x)$ and for all $x\geq0$,
	\be\label{eq:varphi_YW}
	\begin{aligned}
		&\varphi_k^{\prime}(x) = 0 \,\text{ if }\, x \leq a_k, \quad 0 \leq \varphi_k^{\prime}(x) \leq 1 \,\text{ if }\, a_k < x < a_{k-1}, \text{ and }\, \varphi_k^{\prime}(x) = 1 \,\text{ if }\, x \geq a_{k-1};\\
		&0 \leq \varphi_k^{\prime\prime}(x) \leq \frac{2}{k\,\rho_m(x)^2} \,\text{ if }\, a_k < x < a_{k-1}, \quad \varphi_k^{\prime\prime}(x) = 0 \,\text{ otherwise.}
	\end{aligned}
	\ee
	Most importantly, it holds that $\varphi_k(x)\to|x|$ non-decreasingly as $k\to+\infty$ for all $x\in\R$. Since $\varphi_k \in C^2(\R,\R)$ for every $m\in\N$ and $k\in\N$, we can apply It\^o's formula and get:
	\be\label{eq:ito_YW}
	\varphi_k(Z_{t\wedge\tau_m}) = \uI_t + \uII_t + \uIII_t + \uIV_t + \uV_t,
	\ee
	for all $t\in[0,T]$ and $T\in(0,+\infty)$, where we write
	\begin{align*}
		\uI_t &:= K(0)\int_0^{t\wedge\tau_m} {\varphi_k^{\prime}(Z_{s})\ind_{s<\tau} \bigl(\mu(X_{s}) - \mu(Y_{s})\bigr)\ud s},\\
		\uII_t &:= K(0)\int_0^{t\wedge\tau_m} {\varphi_k^{\prime}\bigl(Z_{s}\bigr)\ind_{s<\tau} \bigl(\sigma(X_{s})-\sigma(Y_{s})\bigr)\ud B_s}\\
		&\quad+ \int_0^{t\wedge\tau_m}\ind_{s<\tau}\int_{\,U} {\Bigl(\varphi_k\bigl(Z_{s-} + K(0)\,h\bigl(X_{s-}, Y_{s-} ,u\bigr)\bigr) - \varphi_k\bigl(Z_{s-}\bigr)\Bigr)\tildeN (\ud s, \ud u)},\\
		\uIII_t &:= \frac{1}{2}\,K(0)^2\int_0^{t\wedge\tau_m}{\varphi_k^{\prime\prime}\bigl(Z_{s}\bigr)\ind_{s<\tau} \bigl(\sigma(X_{s})-\sigma(Y_{s})\bigr)^2\ud s},\\
		\uIV_t &:= \int_0^{t\wedge\tau_m} \ind_{s<\tau} \int_{\,U} \Big(\varphi_k\bigl(Z_{s} + K(0)\,h\bigl(X_{s}, Y_{s}, u\bigr)\bigr)\\
		&\quad\quad- \varphi_k\bigl(Z_{s}\bigr) - K(0)\,h\bigl(X_{s}, Y_{s}, u\bigr)\,\varphi_k^{\prime}\bigl(Z_{s}\bigr)\Big) \pi(\ud u)\, \ud s\\
		\uV_t &:= \int_0^{t\wedge\tau_m} {\varphi_k^{\prime}\bigl(Z_s\bigr)\biggl( K'(0) H_s + \int_0^s {K''(s-r) H_r \,\ud r }\biggr)\ud s},
	\end{align*}
	and where we have set $h(x,y,u):=\eta(x,u)-\eta(y,u)$, for all $(x,y,u) \in \R^2 \times U$. Making use of $x\leq|x|$ for all $x\in\R$, \eqref{eq:varphi_YW} and Assumption \ref{ass:YW_conditions}-(i) (we have $|X_{s}|<m$ and $|Y_{s}|<m$ for $s<\tau_m$ by definition of $\tau_m$), we first get 
	\begin{equation*}
		\uI_t \leq \bigl|\uI_t\bigr| \leq K(0)\int_0^{t} {\bigl|\varphi_k^{\prime}\bigl(Z_{s}\bigr)\bigr|\,\ind_{s<\tau_m \wedge \tau}\bigl|\mu(X_{s}) - \mu(Y_{s})\bigr|\,\ud s} \leq K(0)\int_0^t {r_m\bigl(\bigl|Z_{s\wedge\tau_m}\bigr|\bigr)\,\ud s}.
	\end{equation*}
	We can then easily check notably by~\eqref{eq:varphi_YW} and Assumption \ref{ass:conditions1} that $\uII=(\uII_t)_{t\geq0}$ is an $\FF$-martingale and, hence, $\EE[\uII_t]=0$ for all $t\geq0$. We also deal with $\uIII$ through \eqref{eq:varphi_YW} and Assumption \ref{ass:YW_conditions}-(ii),
	\begin{equation*}
		\uIII_t \leq \frac{1}{k}\,K(0)^2\int_0^{t\wedge\tau_m} \ind_{s<\tau}{\frac{\bigl(\sigma(X_{s})-\sigma(Y_{s})\bigr)^2}{\rho_m\bigl(\bigl|Z_{s}\bigr|\bigr)^2}\,\ud s} \leq \frac{1}{k}\,T\,K(0)^2.
	\end{equation*}
	Concerning $\uIV$, we separate the integral over $U$ as follows, for $n\geq1$:
	\begin{align*}
		&\uIV_t = \int_0^{t\wedge\tau_m\wedge \tau}\int_{\{f_m(u)<n\}} \Bigl(\varphi_k\bigl(Z_{s} + K(0)\,h\bigl(X_{s}, Y_{s}, u\bigr)\bigr) - \varphi_k\bigl(Z_{s}\bigr) - K(0)\,h\bigl(X_{s}, Y_{s}, u\bigr)\,\varphi_k^{\prime}\bigl(Z_{s}\bigr)\Bigr)\ud s\,\pi(\ud u)\\
		&\quad+ \int_0^{t\wedge\tau_m\wedge \tau}\int_{\{f_m(u) \geq n\}} \Bigl(\varphi_k\bigl(Z_{s} + K(0)\,h\bigl(X_{s}, Y_{s}, u\bigr)\bigr) - \varphi_k\bigl(Z_{s}\bigr) - K(0)\,h\bigl(X_{s}, Y_{s}, u\bigr)\,\varphi_k^{\prime}\bigl(Z_{s}\bigr)\Bigr)\ud s\,\pi(\ud u)\\
		&=: \uIV_t^1 + \uIV_t^2.
	\end{align*}
	We rewrite $\uIV^1$ with Taylor's formula with integral remainder at order two while injecting \eqref{eq:varphi_YW},
	\begin{align*}
		&\varphi_k\bigl(Z_{s} + K(0)\,h\bigl(X_{s}, Y_{s}, u\bigr)\bigr) - \varphi_k(Z_{s}) - \varphi_k^{\prime}(Z_{s})\,K(0)\,h\bigl(X_{s}, Y_{s}, u\bigr)\\
		&= K(0)^2\,h\bigl(X_{s}, Y_{s}, u\bigr)^2\int_0^1 {(1-r)\,\varphi_k^{\prime\prime}\bigl(Z_{s} + r\,K(0)\,h\bigl(X_{s}, Y_{s}, u\bigr)\bigr)\,\ud r}\\
		&\quad\leq \frac{2}{k}\,K(0)^2\,h\bigl(X_{s}, Y_{s}, u\bigr)^2\int_0^1 {\frac{\ud r}{\rho_m\bigl(\bigl|Z_{s} + r\,K(0)\,h\bigl(X_{s}, Y_{s}, u\bigr)\bigr|\bigr)^2}}.
	\end{align*}
	Since the function $x\mapsto\eta(x, u)$ is non-decreasing for every $u \in U$ by Assumption \ref{ass:YW_conditions}-(ii), we have $Z_{s}\,h(X_{s},Y_{s},u) \geq 0$ almost surely for all $s \in [ 0,\tau)$, notably $|Z_{s} + r\,K(0)\,h(X_{s},Y_{s},u)| \geq |Z_{s}|$. Observing that $\rho_m$ is non-decreasing, it holds that $\rho_m(|Z_{s}|)\leq\rho_m(|Z_{s} + r\,K(0)\,h(X_{s},Y_{s},u)|)$. Then, injecting Assumption \ref{ass:YW_conditions}-(ii) (recall that $|X_{s}|\vee|Y_{s}| \leq m$ for $s<\tau_m$), it follows that
	\begin{equation*}
		\uIV_t^1 \leq \frac{2}{k}\,K(0)^2\int_0^{t}\int_{\{f_m(u)<n\}} {\!\! \ind_{s<\tau_m \wedge \tau}\frac{h\bigl(X_{s},Y_{s},u\bigr)^2}{\rho_m\bigl(\bigl|Z_{s}\bigr|\bigr)^2}\,\ud s\,\pi(\ud u)} \leq \frac{2}{k}\,T\,K(0)^2\,\int_{\{f_m(u)<n\}} {f_m(u)^2\,\pi(\ud u)}.
	\end{equation*}
	In the same vein, using in particular the mean-value theorem, again \eqref{eq:varphi_YW} and Assumption \ref{ass:YW_conditions}-(ii), we bound $\uIV^2$ as follows: 
	\begin{multline*} 
		\uIV_t^2 \leq \bigl|\uIV_t^2\bigr| \leq 2\,K(0)\int_0^{t}\int_{\{f_m(u) \geq n\}} {\ind_{s<\tau_m}\bigl|h\bigl(X_{s}, Y_{s}, u\bigr)\bigr|\,\ud s\,\pi(\ud u)}\\
		\leq 2\,T\,K(0)\,\rho_m(2\,m)\,\int_{\{f_m(u) \geq n\}} {f_m(u)\,\pi(\ud u)},
	\end{multline*}
	where we have also used the fact that  $\rho_m$ is non-decreasing. For the last term, we use~\eqref{eq:varphi_YW} {\red along with Tonelli's theorem} and obtain
	\begin{equation*}
		\uV_t \leq \bigl|\uV_t\bigr| \leq \biggl(\bigl|K^{\prime}(0)\bigr| + {\red\int_0^T {\bigl|K^{\prime\prime}(t)\bigr|\ud t}}\biggr)\int_0^t {\bigl|H_{s\wedge\tau_m}\bigr|\,\ud s}.
	\end{equation*}
	
	In total, injecting all the previously derived bounds into \eqref{eq:ito_YW} and taking the expectation, we obtain using that $r_m$ is concave
	\begin{multline*}
		\EE\bigl[\varphi_k(Z_{t\wedge\tau_m})\bigr] \leq C_{K,T}\biggl(\int_0^t {r_m\bigl(\EE\bigl[\bigl|Z_{s\wedge\tau_m}\bigr|\bigr]\bigr)\,\ud s} + \frac{1}{k}\,\biggl(1 + \int_{\{f_m(u)<n\}} {f_m(u)^2\,\pi(\ud u)}\biggr)\\
		+ \rho_m(2\,m)\int_{\{f_m(u) \geq n\}} {f_m(u)\,\pi(\ud u)} + \int_0^t {\EE\bigl[\bigl|H_{s\wedge\tau_m}\bigr|\bigr]\,\ud s}\biggr),
	\end{multline*}
	where $C_{K,T}\in \R_+$ is a constant depending on the kernel $K$ and $T$ . Letting now $k\to+\infty$ while using Beppo Levi's theorem as $\varphi_k(x)\to|x|$ non-decreasingly for all $x\in\R$, and letting then $n\to+\infty$ (note that $\lim_{n\to+\infty}\,\int_{\{f_m(u) \geq n\}}{f_m(u)\,\pi(\ud u)}=0$ by dominated convergence), we have
	\be\label{eq:eq_Z_YW}
	\EE\bigl[\bigl|Z_{t\wedge\tau_m}\bigr|\bigr] \leq C_{K,T}\biggl(\int_0^t {r_m\bigl(\EE\bigl[\bigl|Z_{s\wedge\tau_m}\bigr|\bigr]\bigr)\,\ud s} + \int_0^t {\EE\bigl[\bigl|H_{s\wedge\tau_m}\bigr|\bigr]\,\ud s}\biggr).
	\ee
	
	In order to derive an inequality for $|Z|+|H|$, we use Proposition~\ref{prop:semi_martingale} and get, since $K(0)>0$, 
	\begin{align}
		\bigl|H_{t\wedge\tau_m}\bigr| &\leq \frac{1}{K(0)}\biggl(\bigl|Z_{t\wedge\tau_m}\bigr| + \left|\int_0^{t\wedge \tau_m} {K^{\prime}(t \wedge \tau_m-s) H_s}\ud s \right|\biggr) \notag\\
		&\leq \frac{1}{K(0)}\biggl(\bigl|Z_{t\wedge\tau_m}\bigr| + \max_{[0,T]}\bigl|K^{\prime}\bigr| \int_0^t {\bigl|H_{s\wedge\tau_m}\bigr|\,\ud s}\biggr).\label{eq:eq_H_YW}
	\end{align}
	Combining therefore \eqref{eq:eq_Z_YW} and  \eqref{eq:eq_H_YW}, we get the following inequality for $|Z|+|H|$,
	\begin{equation*}
		\EE\bigl[\bigl|Z_{t\wedge\tau_m}\bigr| + \bigl|H_{t\wedge\tau_m}\bigr|\bigr] \leq C_{K,T}\int_0^t {\Bigl(r_m\bigl(\EE\bigl[\bigl|Z_{s\wedge\tau_m}\bigr| + \bigl|H_{s\wedge\tau_m}\bigr|\bigr]\bigr) + \EE\bigl[\bigl|Z_{s\wedge\tau_m}\bigr| + \bigl|H_{s\wedge\tau_m}\bigr|\bigr]\Bigr)\ud s},
	\end{equation*}
	for all $t\in[0,T]$ where $T\in(0,+\infty)$ and we recall that $r_m$ is non-decreasing. Thus, by Gr\"onwall's lemma with Osgood's condition (see e.g. Lemma 3.1 of \cite{ChLe}, noting that $\int_0^\varepsilon \frac{1}{x+r_m(x)}dx =+\infty$), we have $\EE[|Z_{t\wedge\tau_m}|]=0$ for all $t\in[0,T]$. Since $\PP(\tau_m\to+\infty,\,m\to+\infty)=1$ and $Z$ is c\`adl\`ag, we get $\PP(Z_t=0, \forall t \in[0,T])=1$. This holds for all $T\in(0,+\infty)$, then $Z$ is null almost surely and, since $Z_t=X_t-Y_t$ for $t\in[0,\tau)$, we deduce that $\PP(X_t=Y_t, \forall t\in[0,\tau))=1$.
\end{proof}

{\red \begin{rem} We see from the estimate of the term $\uV_t$ that we do not need $K''$ to be continuous, but only its local integrability. In fact, if $K'(t)=K'(0)+\int_0^t K''(u) \ud u$ with $ \int_0^t |K''(u)| \ud u<\infty$ for any $t>0$, then the conclusions of Theorem~\ref{thm:pathwise_uniqueness} hold. In fact, all the results of the present paper hold with this weaker assumption, but we prefer to keep in the statements $K\in C^2(\R_+,\R)$ for conciseness since the main difficulty of our results is not there.  
\end{rem}}

\section{Strong existence and non-negative approximation scheme}\label{sec:sec4} 
Set $T\in(0,+\infty)$, $N\in\N^*$, and $t_k := k\,T/N$ for each $k\in\{0,\ldots,N\}$. In this section, we wish to construct a non-negative c\`adl\`ag solution of Equation~\eqref{eq:SVE} on $[0,T]$ by means of two c\`adl\`ag processes: an approximation scheme $\hatX^N = (\hatX_t^N)_{t\in[0,T]}$ and an auxiliary process $\xi^N=(\xi_t^N)_{t\in[0,T]}$. To do so,  we follow the construction proposed by Alfonsi~\cite[Section 3]{Alfonsi23} and adapt it to Equation~\eqref{eq:SVE}. The main difference between~\cite{Alfonsi23} and the present study is that we analyse the convergence in terms of $L^1$ instead of $L^2$ error. This is due to our assumption on the jumps. Indeed, our goal is to extend to Volterra equations the framework of Li and Mytnik~\cite{LM11} for which only first order moments are finite. 

We first recall a useful result for the nonnegativity of the approximating processes $\hatX^N$ and $\xi^N$.
\begin{prop}\label{prop_positif}\cite[Proposition 2.1]{Alfonsi23}
	Let $K:\R_+ \to \R_+$ be a non-increasing kernel that preserves non-negativity such that $0<K(0)<+\infty$. If $x_0\ge 0$, $k\in \N^*$ and $x_1,\dots, x_k \in \R$ are such that 
	$$\forall i \in \{1,\dots, k\}, \  x_0 +\sum_{j=1}^i x_j K(t_i-t_j)\ge 0, $$
	then we have $x_0 +\sum_{j=1}^k x_j \ind_{t_j \le t}  K(t-t_j) \ge 0$ for all $t\ge 0$.
\end{prop}
\noindent Note that this result is true for any discretization grid $0\le t_1<\dots<t_k$, but we state it here directly for convenience on the regular time grid defined above.

We now introduce the approximation schemes:
\begin{enumerate}
	\item[\underline{$k=0$:}] We define $(\hatX_t^N)_{t\in[t_0,t_1)}$ as $\hatX_t^N := X_0$ for $t\in[t_0, t_1)$ and $(\xi_t^N)_{t\in[t_0,t_1)}$ as a c\`adl\`ag solution of 
	\be\label{eq:Fu_and_Li_equation_initial}
	\begin{aligned}
		\xi_t^N &= \hatX_{t_1-}^N + \int_{t_0}^t {K(0)\,\biggl(\mu(\xi_{s}^N)\,\ud s + \sigma(\xi_{s}^N)\,\ud B_s + \int_{\,U} {\eta(\xi_{s-}^N, u)\,\tildeN (\ud s, \ud u)}\biggr)}, 
	\end{aligned}
	\ee
	for $t\in[t_0, t_1)$, where we observe that $\hatX_{t_1-}^N = X_0$.
	\item[\underline{$k=1$:}] We then define $(\hatX_t^N)_{t\in[t_1,t_2)}$ by setting $\hatX_{t_1}^N := \xi_{t_1-}^N$ and 
	\begin{equation*}
		\hatX_t^N := X_0 + \frac{\hatX_{t_1}^N - \hatX_{t_1-}^N}{K(0)}\,K(t-t_1), \qquad t\in[t_1, t_2),
	\end{equation*}
	and $(\xi_t^N)_{t\in[t_1, t_2)}$ as a c\`adl\`ag solution of
	\begin{equation*}
		\xi_t^N = \hatX_{t_2-}^N + \int_{t_1}^t {K(0)\,\biggl(\mu(\xi_{s}^N)\,\ud s + \sigma(\xi_{s}^N)\,\ud B_s + \int_{\,U} {\eta(\xi_{s-}^N, u)\,\tildeN (\ud s, \ud u)}\biggr)}, 
	\end{equation*}
	for $t\in[t_1, t_2)$ where, by continuity of $K$, we have $\hatX_{t_2-}^N = X_0 + \frac{\xi_{t_1-}^N - X_0}{K(0)}K(t_2-t_1)$.
	\item[\underline{$k\geq2$:}] We now assume that we have constructed $(\hatX_t^N)_{t\in[t_0, t_k)}$ and $(\xi_t^N)_{t\in[t_0, t_k)}$, $k<N$, by iteration. In the same vein as above, we set $\hatX_{t_k}^N := \xi_{t_k-}^N$, and define $(\hatX_t^N)_{t\in[t_k, t_{k+1})}$ as
	\be\label{eq:hatX}
	\hatX_t^N := X_0 + \sum_{j=1}^k \frac{\hatX_{t_j}^N - \hatX_{t_j-}^N}{K(0)}\,K(t-t_j), \qquad t\in[t_k, t_{k+1}),
	\ee
	and $(\xi_t^N)_{t\in[t_k, t_{k+1})}$ as a c\`adl\`ag solution of
	\be\label{eq:Fu_and_Li_equation}
	\xi_t^N = \hatX_{t_{k+1}-}^N + \int_{t_k}^t {K(0)\,\biggl(\mu(\xi_{s}^N)\,\ud s + \sigma(\xi_{s}^N)\,\ud B_s + \int_{\,U} {\eta(\xi_{s-}^N, u)\,\tildeN (\ud s, \ud u)}\biggr)},
	\ee
	for $t\in[t_k, t_{k+1})$ where, as before, we have $\hatX_{t_{k+1}-}^N = X_0 + \sum_{1 \leq j \leq k} \frac{\hatX_{t_j}^N - \hatX_{t_j-}^N}{K(0)}K(t_{k+1}-t_j)$.
\end{enumerate}
\begin{rem}
	We observe that when $k=N-1$, we can define $\hatX_{t_N}^N := \hatX_{t_N-}^N$ and $\xi_{t_N}^N := \xi_{t_N-}^N$, since $t_N = T$ is a deterministic time at which no jump can happen almost surely. Thus, $\hatX^N$ and $\xi^N$ are c\`adl\`ag processes on $[0,T]$.
\end{rem}
The next lemma shows, under suitable conditions, that the processes $\hat{X}^N$ and $\xi^N$ are well defined.
\begin{lem}\label{lem:nonnegative_approx}
	Let $X_0\geq0$, Assumptions \ref{ass:conditions1}, \ref{ass:nonneg} and \ref{ass:YW_conditions} hold true, $K : \R_+ \to \R_+$ be non-negative, continuous, non-increasing and non-negativity preserving such that $K(0)>0$. Then, the c\`adl\`ag processes $\hatX^N = (\hatX_t^N)_{t\in[0,T]}$ and $\xi^N=(\xi_t^N)_{t\in[0,T]}$ constructed above are well defined, unique and non-negative:
	\begin{enumerate}
		\item[(i)] for every $k\in\{0,\ldots,N-1\}$, there exists a pathwise unique non-negative c\`adl\`ag solution $(\xi_t^N)_{t\in[t_k, t_{k+1})}$ of Equation~\eqref{eq:Fu_and_Li_equation};
		\item[(ii)] it holds that $\PP(\hatX_t^N \geq 0,\forall t\in[0,T])=1$ and $\PP(\xi_t^N \geq 0,\forall t\in[0,T])=1$.
	\end{enumerate}
\end{lem}
\begin{proof}
	The arguments follow the proof of~\cite[Theorem 3.2]{Alfonsi23}, and we show by induction on $k\in\{1,\ldots,N\}$ that $\PP(\hatX_t^N \geq 0,\forall t\in[0,t_k])=1$ and $\PP(\xi_t^N \geq 0,\forall t\in[0,t_k))=1$ as follows.
	\begin{itemize}
		\item For $k=1$, since $X_0\geq0$, we have trivially $\PP(\hatX_t^N\geq0,\forall t\in[0, t_1))=1$ by construction. Moreover, under Assumptions~\ref{ass:conditions1}, \ref{ass:nonneg} and \ref{ass:YW_conditions}, there exists by \cite[Theorem 2.3]{LM11} a pathwise unique non-negative c\`adl\`ag solution $\xi^N = (\xi_t^N)_{t\in[0, t_1)}$ of Equation~\eqref{eq:Fu_and_Li_equation_initial} with initial value $\hatX_{t_1-}^N = X_0 \geq 0$ almost surely, thus ensuring that $\hatX_{t_1}^N := \xi_{t_1-}^N \geq 0$ almost surely;
		\item Suppose now that $\PP(\hatX_t^N \geq 0,\forall t\in[0,t_k])=1$ for {\red$k\geq1$}. By using Equation~\eqref{eq:hatX}, we write
		\begin{equation*}
			\hatX_{t_i}^N = X_0 + \sum_{j=1}^i \frac{\hatX_{t_j}^N - \hatX_{t_j-}^N}{K(0)}\,K(t_i - t_j) \geq 0, \qquad \text{for all } i\in\{1,\ldots,k\}.
		\end{equation*}
		Since $K$ is non-increasing and preserves non-negativity such that $0 < K(0) < +\infty$, then by Proposition~\ref{prop_positif}, we have $\PP(\hatX_t^N \geq 0,\forall t\in[t_k, t_{k+1}))=1$, notably $\hatX_{t_{k+1}-}^N \geq 0$. Applying again \cite[Theorem 2.3]{LM11}, there exists a pathwise unique non-negative c\`adl\`ag solution $\xi^N = (\xi_t^N)_{t\in[t_k, t_{k+1})}$ of equation \eqref{eq:Fu_and_Li_equation}, yielding $\hatX_{t_{k+1}}^N := \xi_{t_{k+1}-}^N \geq 0$ almost surely.
	\end{itemize}
\end{proof}

By Lemma \ref{lem:nonnegative_approx}, we can exploit the non-negativity of the processes $\hatX^N = (\hatX_t^N)_{t\in[0,T]}$ and $\xi^N=(\xi_t^N)_{t\in[0,T]}$ by working directly with their first-order moments (see \cite[Proposition 2.3]{FL10}).
\begin{prop}\label{prop:first_moment_estimates}
	Let the assumptions of Lemma~\ref{lem:nonnegative_approx} hold and  $\hatX^N = (\hatX_t^N)_{t\in[0,T]}$ and $\xi^N=(\xi_t^N)_{t\in[0,T]}$ be the processes defined therein.  Then, there exists a constant $C_{L,K,T,X_0} \in \R_+$ such that
	\begin{equation*}
		\sup_{N \ge 1}\,{\red\sup_{t\in[0,T]}\,\EE\bigl[\xi_t^N + \hatX_t^N\bigr]} \leq C_{L,K,T,X_0}.
	\end{equation*}
\end{prop}
$\PP(\xi_t^N \geq 0,\forall t\in[t_k, t_{k+1}))=1$ for every $k\in\{0,\ldots,N-1\}$ by virtue of Lemma~\ref{lem:nonnegative_approx}
\begin{proof}
	For every $m\geq1$, define $\tau_m := \inf\{t \in [ 0 , T] : \xi_t^N \geq m\}$  and $\tau_m^k:=\tau_m \vee t_k$. Since $\xi^N$ is c\`adl\`ag and $\FF$-adapted by construction, $\tau_m$ is an $\FF$-stopping time for every $m\geq1$. {\red From Equation \eqref{eq:Fu_and_Li_equation}, we get for $t\in[t_k,t_{k+1})$
	\begin{multline*} 
		\EE\bigl[\xi_{t\wedge\tau^k_m}^N\bigl|\cF_{t_k}\bigr] = \hatX_{t_{k+1}-}^N + \EE\biggl[\int_{t_k}^{t\wedge\tau^k_m} {\!K(0)\,\mu(\xi_{s}^N)\,\ud s}\,\biggl|\cF_{t_k}\biggr] + \EE\biggl[\int_{t_k}^{t\wedge\tau_m^k} {\!K(0)\,\sigma(\xi_{s}^N)\,\ud B_s}\,\biggl|\cF_{t_k}\biggr]\\
		+ \EE\biggl[\int_{t_k}^{t\wedge\tau_m^k}\!\int_{\,U} {K(0)\,\eta(\xi_{s-}^N, u)\,\tildeN (\ud s, \ud u)}\,\biggl|\cF_{t_k}\biggr].
	\end{multline*}
	Note that these conditional expectations are well defined since $\PP(\xi_t^N \geq 0,\forall t\in[t_k, t_{k+1}))=1$ for every $k\in\{0,\ldots,N-1\}$ by virtue of Lemma~\ref{lem:nonnegative_approx}. Besides, it can be easily checked that the last two conditional expectations of the right-hand side are null since both integrated processes are $\FF$-martingales under Assumption \ref{ass:conditions1}. Using again Assumption~\ref{ass:conditions1} and the non-negativity of $\xi$, we have
	\begin{equation*}
		\EE\bigl[\xi_{t\wedge\tau_m^k}^N\bigl|\cF_{t_k}\bigr] \leq \hatX_{t_{k+1}-}^N +K(0)\,L\int_{t_k}^t {\bigl(1 + \EE\bigl[\xi_{s\wedge\tau_m^k}^N\bigl|\cF_{t_k}\bigr]\bigr)\ud s}.
	\end{equation*}
	After applying Gr\"onwall's lemma and taking the expectation, this yields
	\begin{equation*}
		1 + \EE\bigl[\xi_{t\wedge\tau_m^k}^N\bigr] \leq \bigl(1+\EE\bigl[\hatX_{t_{k+1}-}^N\bigr]\bigr)e^{K(0)L\,T}.
	\end{equation*}
	Since $\xi$ is a c\`adl\`ag process, we have $\tau_m\to +\infty$ almost surely as $m\to +\infty$ and thus $\tau_m^k\to +\infty$. Using then Fatou's lemma and taking the supremum, we obtain
	\be\label{eq:first_moment_xi}
	1+\sup_{t\in[t_k, t_{k+1})}\!\EE\bigl[\xi_t^N\bigr] \leq \bigl(1 + \EE\bigl[\hatX_{t_{k+1}-}^N\bigr]\bigr)e^{K(0)L\,T}.
	\ee}
	From this, we easily have by induction on $k$ that $\EE[\xi^N_t]<\infty$ and $\EE[\hatX_{t}^N]<\infty$ for all $t\in[t_0,t_k)$ and thus for $t\in[0,T]$. 
	Recall that by continuity of $K$, we have for every $k\in\{0,\ldots,N-1\}$,
	\begin{equation*}
		\hatX_{t_{k+1}-}^N = X_0 + \sum_{j=1}^k \frac{\hatX_{t_j}^N - \hatX_{t_j-}^N}{K(0)}\,K(t_{k+1}-t_j).
	\end{equation*}
	Using $\hatX_{t_j}^N = \xi_{t_j-}^N$, $\hatX_{t_j-}^N = \xi_{t_{j-1}}^N$,  Assumption~\ref{ass:conditions1} and~\eqref{eq:first_moment_xi}, we obtain  
	\begin{align*}
		\EE\bigl[\hatX_{t_{k+1}-}^{N}\bigr] &= X_0 + \sum_{j=1}^k K(t_{k+1}-t_j)\,\EE\biggl[\int_{t_{j-1}}^{t_j} {\mu(\xi_{s-}^N)\,\ud s}\biggr]\\
		&\quad\leq X_0 + L\,\left(\max_{[0,T]}K \right)\,\frac{T}{N}\,\sum_{j=1}^k \Bigl(1 + {\red\sup_{t\in[t_{j-1}, t_j)}\!\EE\bigl[\xi_t^N\bigr]}\Bigr)\\
		&\quad\leq X_0 + L\,e^{K(0)L\,T} \left(\max_{[0,T]}K\right)\,\frac{T}{N}\,\sum_{j=1}^k \bigl(1 + \EE\bigl[\hatX_{t_j-}^N\bigr]\bigr),
	\end{align*}
	where we again used $x\leq|x|$ for all $x\in\R$, Assumption \ref{ass:conditions1}, and  \eqref{eq:first_moment_xi}. Thus, using a discrete version of Gr\"onwall's lemma (see, e.g., \cite{Clark87}), where we denote $C_{L,K,T}:= L\,\left(\max_{[0,T]}\!K\right)\,Te^{K(0)L\,T}$,
	\begin{align}
		1 + \max_{1\le k\le N}\EE\bigl[\hatX_{t_k-}^N\bigr] &\leq \bigl(1 + X_0\bigr)\Bigl(1 + \frac{C_{L,K,T}}{N}\Bigr)^N\notag\\
		&\leq \bigl(1 + X_0\bigr)e^{C_{L,K,T}}.\label{eq:first_moment_hatX}
	\end{align}
	Injecting then \eqref{eq:first_moment_hatX} into \eqref{eq:first_moment_xi}, we get by iteration over $k\in\{0,\ldots,N-1\}$,
	\begin{equation*}
		1 + {\red\sup_{t\in[0,T]}\EE\bigl[\xi_t^N\bigr]} \leq \bigl(1 + X_0\bigr)e^{K(0)L\,T + C_{L,K,T}}.
	\end{equation*}
	In the same fashion as above, we have for $t\in[t_k,t_{k+1})$
	\begin{align*}
		\EE\bigl[\hatX_{t}^{N}\bigr] &= X_0 + \sum_{j=1}^k K(t-t_j)\,\EE\biggl[\int_{t_{j-1}}^{t_j} {\mu(\xi_{s-}^N)\,\ud s}\biggr]\leq X_0 + L\,e^{K(0)L\,T}\left(\max_{[0,T]}K\right)\,\frac{T}{N}\,\sum_{j=1}^k \bigl(1 + \EE\bigl[\hatX_{t_j-}^N\bigr]\bigr),
	\end{align*}
	and thus
	\begin{equation*}
		{\red\sup_{t\in[0,T]}\EE\bigl[\hatX_t^N\bigr]} \leq X_0 + \bigl(1 + X_0\bigr)C_{L,K,T}\,e^{C_{L,K,T}}.\qedhere
	\end{equation*}
\end{proof}
The next two lemmata are technical results in order to state Proposition \ref{prop:barX} thereafter. In doing so, let us denote by $w_{K,T}(\delta)$, for all $\delta>0$, the modulus of continuity of $K$, which is given by
\begin{equation*}
	w_{K,T}(\delta) := \max\bigl\{\bigl|K(t)-K(s)\bigr|:(s,t)\in[0,T]^2,\,|t-s|\leq\delta\bigr\}.
\end{equation*}
\begin{lem}\label{lem:first_moment_estimates_bis} Let the assumptions of Lemma~\ref{lem:nonnegative_approx} hold and  $\hatX^N = (\hatX_t^N)_{t\in[0,T]}$ and $\xi^N=(\xi_t^N)_{t\in[0,T]}$ be the processes defined therein. 
	Then, there exists a constant $C_{L,K,T,X_0} \in \R_+$ such that
	\begin{equation*}
		{\red\sup_{t\in[0,T]}\EE\bigl[\bigl|\xi_t^N - \hatX_t^N\bigr|\bigr]} \leq C_{L,K,T,X_0} \sqrt{\frac{T}{N}}\Biggl(1 + N\,w_{K,T}\biggl(\frac{T}{N}\biggr)\Biggr).
	\end{equation*}
\end{lem}
\begin{proof}
	For every $k\in\{0,\ldots,N-1\}$, for all $t\in[t_k, t_{k+1})$, we make use of \eqref{eq:hatX} and \eqref{eq:Fu_and_Li_equation} as follows:
	\begin{align*}
		\bigl|\xi_t^N - \hatX_t^N\bigr| &\leq \bigl|\xi_t^N - \hatX_{t_{k+1}-}^N\bigr| + \bigl|\hatX_{t_{k+1}-}^N - \hatX_t^N\bigr|\\
		&\leq K(0)\biggl|\int_{t_k}^t {\mu(\xi_{s}^N)\,\ud s} + \int_{t_k}^t {\sigma(\xi_{s}^N)\,\ud B_s} + \int_{t_k}^{t}\int_{\,U} {\eta(\xi_{s-}^N, u)\,\tildeN (\ud s, \ud u)}\biggr|\\
		&\quad+ \sum_{j=1}^k\,\Bigl|K(t_{k+1} - t_j) - K(t-t_j)\Bigr|\,\Biggl|\frac{\hatX_{t_j}^N-\hatX_{t_j-}^N}{K(0)}\Biggr|\\
		&\leq K(0)\biggl|\int_{t_k}^t {\mu(\xi_{s}^N)\,\ud s} + \int_{t_k}^t {\sigma(\xi_{s}^N)\,\ud B_s} + \int_{t_k}^{t}\int_{\,U} {\eta(\xi_{s-}^N, u)\,\tildeN (\ud s, \ud u)}\biggr|\\
		&\quad+ w_{K,T}\biggl(\frac{T}{N}\biggr)\sum_{j=1}^k\, \Biggl|\int_{t_{j-1}}^{t_j} {\mu(\xi_{s}^N)\,\ud s} + \int_{t_{j-1}}^{t_j} {\sigma(\xi_{s}^N)\,\ud B_s} + \int_{t_k}^{t}\int_{\,U} {\eta(\xi_{s-}^N, u)\,\tildeN (\ud s, \ud u)} \Biggr|.
	\end{align*}
	Then, for every $1\le j\le k+1$, we get by Proposition~\ref{prop:prem_first_moment_estimate}
	\begin{align*} 	&\EE\biggl[\biggl|\int_{t_j}^{t_{j+1}\wedge t} {\mu(\xi_{s}^N)\,\ud s} + \int_{t_j}^{t_{j+1}\wedge t} {\sigma(\xi_{s}^N)\,\ud B_s} + \int_{t_j}^{t_{j+1}\wedge t}\int_{\,U} {\eta(\xi_{s-}^N, u)\,\tildeN (\ud s, \ud u)}\biggr|\biggr]\\ 
		& \le C_L 
		\left( \frac T N + \int_{t_j}^{t_{j+1}\wedge t} {\EE\bigl[\xi_s^N \bigr]\ud s}  +  \biggl( \frac T N + \int_{t_j}^{t_{j+1}\wedge t} {\EE\bigl[\xi_s^N \bigr]\ud s} \biggr)^{1/2} \right) \\	&\leq C_L(1+\tilde{C}_{L,K,T,X_0})(\sqrt{T}+1)\sqrt{\frac T N},	\end{align*}
	where $\tilde{C}_{L,K,T,X_0}$ is the constant given by Proposition~\ref{prop:first_moment_estimates} which upper bounds $\EE[\xi_t^N]$. We set $C_{L,K,T,X_0}=C_L(1+\tilde{C}_{L,K,T,X_0})(1+\sqrt{T})$ and get
	\begin{align*}
		\EE\bigl[\bigl|\xi_t^N - \hatX_t^N\bigr|\bigr] &\leq C_{L,K,T,X_0}\left( K(0) + N w_{K,T}\biggl(\frac{T}{N}\biggr) \right) \sqrt{\frac T N},
	\end{align*}
	which gives the claim.
\end{proof}

When $K \in C^1(\R_+,\R_+)$, $Nw_{K,T}(T/N)$ is uniformly bounded in~$N$ and Lemma~\ref{lem:first_moment_estimates_bis} indicates that the two approximating processes $\xi^N$ and {\red$\hatX^N$} are close when $N$ gets large.  We now introduce  a third approximating process that will be more convenient to use with It\^o calculus.  
For every $N\geq1$, let $\nu(\cdot,N):[0,T]\to\{0,\cdots,N-1\}$ be such that $\nu(T,N) := N-1$ and for every $k\in\{0,\cdots,N-1\}$ and for all $t\in[t_k, t_{k+1})$, $\nu(t,N) := k$. We now rewrite $\hatX_t^N$ for $t\in[t_k, t_{k+1})$ as  
\begin{align}
	\hatX_t^N &= X_0 + \sum_{j=1}^k \int_{t_{j-1}}^{t_j} {\!K(t - t_j)\biggl(\mu(\xi_{s}^N)\,\ud s + \sigma(\xi_{s}^N)\,\ud B_s + \int_{\,U} {\eta(\xi_{s-}^N, u)\,\tildeN (\ud s, \ud u)}\biggr)}\notag\\
	&= X_0 + \int_0^{t_{\nu(t,N)}} {K\bigl(t - t_{\nu(s,N)+1}\bigr)\biggl(\mu(\xi_{s}^N)\,\ud s + \sigma(\xi_{s}^N)\,\ud B_s + \int_{\,U} {\eta(\xi_{s-}^N, u)\,\tildeN (\ud s, \ud u)}\biggr)}.\label{eq:nu_N}
\end{align}
Let us then define the process $\barX^N=(\barX_t^N)_{t\in[0,T]}$ as
\be\label{eq:barX}
\barX_t^N := X_0 + \int_0^t {K(t-s)\biggl(\mu(\xi_{s}^N)\,\ud s + \sigma(\xi_{s}^N)\,\ud B_s + \int_{\,U} {\eta(\xi_{s-}^N, u)\,\tildeN (\ud s, \ud u)}\biggr)}.
\ee
Note that the process $\bar{X}^N$ may not be non-negative. However, comparing~\eqref{eq:barX} and~\eqref{eq:nu_N}, we may expect it to be close to the non-negative process $\hat{X}^N$. This is stated in the following lemma. 

\begin{lem}\label{lem:barX} Let the assumptions of Lemma~\ref{lem:nonnegative_approx} hold, $\hatX^N$ and $\xi^N$ be the processes defined therein, and $\bar{X}^N$ be defined by~\eqref{eq:barX}.  Then, there exists a constant $C_{L,K,T,X_0} \in \R_+$ such that
	\begin{align*}
		{\red\sup_{t\in[0,T]}\EE\bigl[\bigl|\hatX_t^N - \barX_t^N\bigr|\bigr]} &\leq C_{L,K,T,X_0}\Biggl( \sqrt{\frac{T}{N}} + w_{K,T}\biggl(\frac{T}{N}\biggr)\Biggr),\\ 
		{\red\sup_{t\in[0,T]}\EE\bigl[\bigl|\xi_t^N - \barX_t^N\bigr|\bigr]} &\leq C_{L,K,T,X_0}\,\Biggl( \sqrt{\frac{T}{N}} + w_{K,T}\biggl(\frac{T}{N}\biggr) + \sqrt{\frac{T}{N}} \times N\, w_{K,T}\biggl(\frac{T}{N}\biggr) \Biggr).
	\end{align*}
	If in addition $K \in C^1(\R_+,\R_+)$, then there exists a constant $C_{L,K,T,X_0}\in \R_+$ such that
	\be\label{eq:lem articular_case}
	{\red\sup_{t\in[0,T]}\EE\bigl[\bigl|\hatX_t^N - \barX_t^N\bigr|\bigr] +\sup_{t\in[0,T]}\EE\bigl[\bigl|\xi_t^N - \barX_t^N\bigr|\bigr]} \leq C_{L,K,T,X_0}\,\frac{1}{\sqrt{N}}.
	\ee
\end{lem}
\begin{proof}
	For all $t\in[0,T]$, using \eqref{eq:nu_N} and \eqref{eq:barX}, we write
	\begin{align*}
		\bigl|\hatX_t^N - \barX_t^N\bigr| &\leq \biggl|\int_0^{t_{\nu(t,N)}} {\!\!\Bigl(K\bigl(t - t_{\nu(s,N)+1}\bigr)-K\bigl(t - s\bigr)\Bigr)\!\biggl(\mu(\xi_{s}^N)\,\ud s + \sigma(\xi_{s}^N)\,\ud B_s + \int_{\,U} {\eta(\xi_{s-}^N, u)\,\tildeN (\ud s, \ud u)}\biggr)}\biggr|\\
		&\quad+ \biggl|\int_{t_{\nu(t,N)}}^t {\!\!K(t-s)\biggl(\mu(\xi_{s}^N)\,\ud s + \sigma(\xi_{s}^N)\,\ud B_s + \int_{\,U} {\eta(\xi_{s-}^N, u)\,\tildeN (\ud s, \ud u)}\biggr)}\biggr|\\
		&=: \uI + \uII.
	\end{align*}
	We take the expectation of $\uI$ and use Propositions~\ref{prop:first_moment_estimates} and~\ref{prop:prem_first_moment_estimate} with $p=0$, $q=\tau=t_{\nu(t,N)}$, and $H(t,s):=\bigl(K\bigl(t - t_{\nu(s,N)+1}\bigr)-K\bigl(t - s\bigr)\bigr)\ind_{\{s \leq t_{\nu(t,N)}\}}$, $t \leq T$, noticing that $\lVert H \rVert_t \leq w_{K,T}(T/N)$ to get
	\begin{equation*}
		\EE\bigl[\uI\bigr] \leq \tilde{C}_{L,K,T,X_0} w_{K,T}(T/N),
	\end{equation*}
	where $\tilde{C}_{L,K,T,X_0} \in \R_+$ is a constant depending only on $L$, the kernel~$K$, $T$ and $X_0$.
	Taking then the expectation of $\uII$, and using again Proposition~\ref{prop:first_moment_estimates} and~\ref{prop:prem_first_moment_estimate} with $p=t_{\nu(t,N)}$, $q=\tau=t$, $H(t,s):=K(t-s)\ind_{\{t_{\nu(t,N)} \leq s \leq t\}}$, $t \leq T$, with $\lVert H \rVert_t \leq \max_{[0,T]}K$ and $|p-q| \leq T/N$, we have
	\begin{equation*}
		\EE\bigl[\uII\bigr] = \EE\Bigl[\EE\bigl[\uII\,\bigl|\cF_p\bigr]\Bigr] \leq \tilde{C}_{L,K,T,X_0} \sqrt{\frac{T}{N}},
	\end{equation*}
	where we have made use of the law of iterated expectations and Jensen's inequality. The second inequality then follows by the triangle inequality {\red$\sup_{t\in[0,T]}\EE[|\xi_t^N - \barX_t^N|] \leq \sup_{t\in[0,T]}\EE[|\xi_t^N - \hatX_t^N|] + \sup_{t\in[0,T]}\EE[|\hatX_t^N - \barX_t^N|]$} and Lemma~\ref{lem:first_moment_estimates_bis}. When $K \in C^1(\R_+,\R_+)$, we have $w_{K,T}(T/N)\leq\max_{[0,T]}|K^{\prime}|\,T/N$ and thus~\eqref{eq:lem articular_case}.
\end{proof}

We are now in position to prove our strong existence result by showing the convergence of the approximating processes. We consider henceforth the approximating processes on $[0,T]$ associated respectively to the regular discretization grids of time steps $T/M$ and $T/N$, with $M,N\in \N^*$. Thus, we have at hand the processes $\hatX^M$ and $\hatX^N$, $\xi^M$ and $\xi^N$, along with $\barX^M$ and $\barX^N$. To upper bound $\EE[|\barX_t^M - \barX_t^N|]$, we need to introduce the following global assumption, which can be seen as a global version of Assumption~\ref{ass:local_conditions}.

\begin{ass}\label{ass:global_conditions}
	Suppose that
	\begin{enumerate}
		\item[(i)] there exists a constant $L^{\prime} > 0$ such that 
		\begin{equation*}
			\bigl|\mu(x) - \mu(y)\bigr| + \bigl|\sigma(x) - \sigma(y)\bigr|^2 \leq L^{\prime}\,\bigl|x - y\bigr|, \qquad \text{for all } (x,y)\in\R^2;
		\end{equation*}
		\item[(ii)]  the function $x\mapsto\eta(x,u)$ is non-decreasing for every $u \in U$ and there exists a non-negative Borel function $f:U\to\R_+$ such that 
		\begin{equation*}
			\bigl|\eta(x, u) - \eta(y, u)\bigr| \leq \bigl|x - y\bigr|^{1/2}f(u), \qquad \text{for all } (x,y,u) \in \R^2 \times U,
		\end{equation*}
		where $f$ satisfies $\int_{\,U} {(f(u) \wedge f(u)^2)\,\pi(\ud u)} < +\infty$.
	\end{enumerate}
\end{ass} 

{\red\begin{rem}
	In contrast with Assumption \ref{ass:YW_conditions}, we impose here $r_m(t)=\rho_m(t)^2=L^{\prime}\,t$, $t\in\R_+$. In fact, we need Yamada--Watanabe functions with further properties to deal with the approximation error in the next proposition. In particular, this special choice is important to have the last estimate of Equation \eqref{eq:varphi_existence} below. This was also pointed by \cite[Remark 2.4]{PS23}.
\end{rem}}

\begin{prop}\label{prop:barX} Suppose that $X_0\geq0$, Assumptions \ref{ass:conditions1}, \ref{ass:nonneg} and \ref{ass:global_conditions} hold true, $K \in C^2(\R_+,\R_+)$ is non-increasing, preserves non-negativity and such that $K(0)>0$.
	For $M,N>1$, $M \neq N$, let $\barX^M=(\barX_t^M)_{t\in[0,T]}$ and $\barX^N=(\barX_t^N)_{t\in[0,T]}$ be defined by~\eqref{eq:barX}.  Then, there exists a constant $C_{L,L^{\prime},K,T,f,X_0} \in \R_+$ such that
	\begin{equation*}
		{\red\sup_{t\in[0,T]}\EE\bigl[\bigl|\barX_t^M - \barX_t^N\bigr|\bigr]} \leq C_{L,L^{\prime},K,T,f,X_0}\,\frac{1}{\log(M \wedge N)}.
	\end{equation*}
\end{prop}
\begin{proof}
	We begin by approximating the absolute value by suitable smooth functions $\varphi_{\delta,\varepsilon} \in C^2(\R,\R_+)$, where $\varepsilon \in (0,1)$ and $\delta \in (1,+\infty)$, such that for every $\varepsilon$ and $\delta$, $\varphi_{\delta,\varepsilon}$ satisfies
	\be\label{eq:varphi_existence}  
	|x| \leq \varepsilon + \varphi_{\delta,\varepsilon}(x), \quad 0 \leq |\varphi_{\delta,\varepsilon}^{\prime}(x)| \leq 1 \quad \text{and} \quad 0 \leq \varphi_{\delta,\varepsilon}^{\prime\prime}(x) \leq \frac{2}{|x|\log\delta}\,\ind_{[\varepsilon/\delta, \varepsilon]}(|x|), 
	\ee
	for all $x\in\R$ (refer to Appendix \ref{sec:appendix} for further details). Since $K \in C^2(\R_+,\R_+)$ by hypothesis, $\barX^{M,N}:=\barX^M-\barX^N$ is an $\FF$-semimartingale by Proposition \ref{prop:semi_martingale}, which we can express as 
	\begin{equation*}
		\barX_t^{M,N} = K(0)\,\Xi_t^{M,N} + K'(0)  \int_0^t \Xi_s^{M,N} \ud s +  \int_0^t {\biggl(\int_0^s {K''(s-r) \Xi_r^{M,N} \,\ud r}\biggr)\ud s},
	\end{equation*}
	where we define the process $\Xi^{M,N}=(\Xi_t^{M,N})_{t\in[0,T]}$ as
	\begin{equation*}
		\Xi_t^{M,N} := \int_0^t {\biggl(\bigl(\mu(\xi_{s}^M) - \mu(\xi_{s}^N)\bigr)\ud s + \bigl(\sigma(\xi_{s}^M)-\sigma(\xi_{s}^N)\bigr)\ud B_s + \int_{\,U} {\bigl(\eta(\xi_{s-}^M, u)-\eta(\xi_{s-}^N, u)\bigr)\tildeN (\ud s, \ud u)}\biggr)}.
	\end{equation*}
	By relying on \eqref{eq:varphi_existence} and applying It\^o's formula to $\varphi_{\delta,\varepsilon}(\barX^{M,N})$, we have
	\be\label{eq:eq_barX}
	\bigl|\barX_t^{M,N}\bigr| \leq \varepsilon + \varphi_{\delta,\varepsilon}\bigl(\barX_t^{M,N}\bigr) = \varepsilon + \uI_t + \uII_t + \uIII_t + \uIV_t + \uV_t,
	\ee
	for all $t\in[0,T]$, where we write
	\begin{align*}
		\uI_t &:= K(0)\int_0^t {\varphi_{\delta,\varepsilon}^{\prime}\bigl(\barX_s^{M,N}\bigr)\bigl(\mu(\xi_s^M) - \mu(\xi_s^N)\bigr)\ud s},\\
		\uII_t &:= K(0)\int_0^t {\varphi_{\delta,\varepsilon}^{\prime}\bigl(\barX_s^{M,N}\bigr)\bigl(\sigma(\xi_s^M)-\sigma(\xi_s^N)\bigr)\ud B_s}\\
		&\quad+ \int_0^t\int_{\,U} {\Bigl(\varphi_{\delta,\varepsilon}\bigl(\barX_{s-}^{M,N} + K(0)\,h\bigl(\xi_{s-}^M, \xi_{s-}^N, u\bigr)\bigr) - \varphi_{\delta,\varepsilon}\bigl(\barX_{s-}^{M,N}\bigr)\Bigr)\tildeN (\ud s, \ud u)},\\
		\uIII_t &:= \frac{1}{2}\,K(0)^2\int_0^t{\varphi_{\delta,\varepsilon}^{\prime\prime}\bigl(\barX_s^{M,N}\bigr)\bigl(\sigma(\xi_s^M)-\sigma(\xi_s^N)\bigr)^2\ud s},\\
		\uIV_t &:= \int_0^t\int_{\,U} \Big(\varphi_{\delta,\varepsilon}\bigl(\barX_s^{M,N} + K(0)\,h\bigl(\xi_s^M, \xi_s^N, u\bigr)\bigr)\\
		&\quad\quad- \varphi_{\delta,\varepsilon}\bigl(\barX_s^{M,N}\bigr) - K(0)\,h\bigl(\xi_s^M, \xi_s^N, u\bigr)\,\varphi_{\delta,\varepsilon}^{\prime}\bigl(\barX_s^{M,N}\bigr)\Big)\ud s\,\pi(\ud u)\\
		\uV_t &:= \int_0^t {\varphi_{\delta,\varepsilon}^{\prime}\bigl(\barX_s^{M,N}\bigr)\biggl(K'(0)\,\Xi_s^{M,N} + \int_0^s {K''(s-r) \Xi_r^{M,N} \,\ud r}\biggr)\ud s},
	\end{align*}
	and where we have set $h(x,y,u):=\eta(x,u)-\eta(y,u)$, for all $(x,y,u) \in \R^2 \times U$. Making use of $x\leq|x|$ for all $x\in\R$, \eqref{eq:varphi_existence} and Assumption \ref{ass:global_conditions}-(i), we first get 
	\begin{align*}
		\uI_t \leq \bigl|\uI_t\bigr| &\leq K(0)\int_0^t {\bigl|\varphi_{\delta,\varepsilon}^{\prime}\bigl(\barX_s^{M,N}\bigr)\bigr|\,\bigl|\mu(\xi_s^M) - \mu(\xi_s^N)\bigr|\,\ud s}\\
		&\leq K(0)\,L^{\prime}\biggl(\int_0^t {\bigl|\barX_s^M - \xi_s^M\bigr|\,\ud s} + \int_0^t {\bigl|\barX_s^{M,N}\bigr|\,\ud s} + \int_0^t {\bigl|\barX_s^N - \xi_s^N\bigr|\,\ud s}\biggr).
	\end{align*}
	We can then easily check notably by means of Assumption \ref{ass:conditions1}, \eqref{eq:varphi_existence} and Proposition \ref{prop:first_moment_estimates}, that $\uII=(\uII_t)_{t\in[0,T]}$ is an $\FF$-martingale and, hence, $\EE[\uII_t]=0$ for all $t\in[0,T]$. Subsequently, we deal with $\uIII$ by using \eqref{eq:varphi_existence} and Assumption \ref{ass:global_conditions}-(i),
	\begin{align*}
		\uIII_t &\leq K(0)^2\frac{L^{\prime}}{\log\delta}\int_0^t {\ind_{[\varepsilon/\delta, \varepsilon]}\bigl(\bigl|\barX_s^{M,N}\bigr|\bigr)\frac{\bigl|\xi_s^M - \xi_s^N\bigr|}{\bigl|\barX_s^{M,N}\bigr|}\,\ud s}\\
		&\leq K(0)^2\frac{\delta\,L^{\prime}}{\varepsilon\,\log\delta}\biggl(\frac{\varepsilon\,T}{\delta} + \int_0^t {\bigl|\barX_s^M - \xi_s^M\bigr|\,\ud s} + \int_0^t {\bigl|\barX_s^N - \xi_s^N\bigr|\,\ud s}\biggr).
	\end{align*}
	Concerning $\uIV$, it can be rewritten as follows
	\begin{align*}
		\uIV_t &= \int_0^t\int_{\,U} \Big(\varphi_{\delta,\varepsilon}\bigl(\barX_s^{M,N} + K(0)\,h\bigl(\barX_s^M, \barX_s^N, u\bigr)\bigr) - \varphi_{\delta,\varepsilon}\bigl(\barX_s^{M,N}\bigr)\\
		&\quad\quad- K(0)\,h\bigl(\barX_s^M, \barX_s^N, u\bigr)\,\varphi_{\delta,\varepsilon}^{\prime}\bigl(\barX_s^{M,N}\bigr)\Big)\ud s\,\pi(\ud u)\\
		&\quad+ \int_0^t\int_{\,U} \Big(\varphi_{\delta,\varepsilon}\bigl(\barX_s^{M,N} + K(0)\,h\bigl(\xi_s^M, \xi_s^N, u\bigr)\bigr) - \varphi_{\delta,\varepsilon}\bigl(\barX_s^{M,N} + K(0)\,h\bigl(\barX_s^M, \barX_s^N, u\bigr)\bigr)\\
		&\quad\quad- K(0)\,\Bigl(h\bigl(\xi_s^M, \xi_s^N, u\bigr) - h\bigl(\barX_s^M, \barX_s^N, u\bigr)\Bigr)\,\varphi_{\delta,\varepsilon}^{\prime}\bigl(\barX_s^{M,N}\bigr)\Big)\ud s\,\pi(\ud u).
	\end{align*}
	Using Assumption \ref{ass:global_conditions}-(ii) and Lemmata~\ref{lem:lem_varphi_1} and ~\ref{lem:lem_varphi_2}  with $x=\barX_s^M$, $y=\barX_s^N$, $z=\barX_s^{M,N}$, $\alpha=\xi_s^M$, $\beta=\xi_s^N$ and $c=K(0)$, we get 
	\begin{multline*}
		\uIV_t \leq \bigl(K(0) \vee K(0)^2\bigr) \int_{\,U}{\bigl(f(u) \wedge f(u)^2\bigr)\pi(\ud u)}\,\biggl[ 2\,T  \left(\varepsilon^{1/2} + \frac{1}{\log\delta}\right) \\+ 6 \int_0^t{\bigl|\barX_s^M-\xi_s^M\bigr|^{1/2}+\bigl|\barX_s^N-\xi_s^N\bigr|^{1/2}\,\ud s}  
		+  \frac{6}{\log \delta} + \frac{ 6 \delta} {\varepsilon\log\delta}\biggl(  \int_0^t{\bigl|\bar{X}^M-\xi_s^M|+|\bar{X}^N_s-\xi_s^N\bigr|\,\ud s}\biggr)\biggr].
	\end{multline*}
	For the last term, we use~\eqref{eq:varphi_existence} {\red along with Tonelli's theorem} and have
	\begin{equation*}
		\uV_t \leq \bigl|\uV_t\bigr| \leq \biggl(\bigl|K^{\prime}(0)\bigr| + {\red\int_0^T {\bigl|K^{\prime\prime}(t)\bigr|\ud t}}\biggr)\int_0^t {\bigl|\Xi_s^{M,N}\bigr|\,\ud s}.
	\end{equation*}
	In order to derive an inequality for $|\barX^{M,N}|+|\Xi^{M,N}|$, we go back to \eqref{eq:barX} so as to express $\barX^{M,N}$ with respect to $\Xi^{M,N}$ as follows:
	\begin{equation*}
		\barX_t^{M,N} = \int_0^t {K(t-s)\,\ud\Xi_s^{M,N}}.
	\end{equation*}
	By Proposition~\ref{prop:semi_martingale}, we get
	\begin{equation*}
		\barX_t^{M,N} = K(0)\,\Xi_t^{M,N} + \int_0^t {K^{\prime}(t-s)\,\Xi_s^{M,N}\,\ud s},
	\end{equation*}
	and, since $K(0)>0$, we can write
	\be\label{eq:eq_Xi}
	\bigl|\Xi_t^{M,N}\bigr| \leq \frac{1}{K(0)}\biggl(\bigl|\barX_t^{M,N}\bigr| + \max_{[0,T]}\bigl|K^{\prime}\bigr|\int_0^t {\bigl|\Xi_s^{M,N}\bigr|\,\ud s}\biggr).
	\ee
	In total, adding all the previously derived inequalities while combining \eqref{eq:eq_barX} with \eqref{eq:eq_Xi}, taking also the expectation (all quantities are non-negative), and using Jensen's inequality, we obtain
	\begin{align*}
		\EE\bigl[\bigl|\barX_t^{M,N}\bigr| + \bigl|\Xi_t^{M,N}\bigr|\bigr] &\leq C_{L^{\prime},K,T,f}\int_0^t {\EE\bigl[\bigl|\barX_s^{M,N}\bigr| + \bigl|\Xi_s^{M,N}\bigr|\bigr]\ud s}\\
		&\quad+ C_{L^{\prime},K,T,f}\bigg(\varepsilon + \varepsilon^{1/2} + \int_0^t {\EE\bigl[\bigl|\barX_s^M - \xi_s^M\bigr|\bigr]^{1/2}\ud s} + \int_0^t {\EE\bigl[\bigl|\barX_s^N - \xi_s^N\bigr|\bigr]^{1/2}\ud s}\\
		&\quad+ \frac{1}{\log\delta} + \biggl(1 + \frac{\delta}{\varepsilon\log\delta}\biggr)\biggl(\int_0^t {\EE\bigl[\bigl|\barX_s^M - \xi_s^M\bigr|\bigr]\ud s} + \int_0^t {\EE\bigl[\bigl|\barX_s^N - \xi_s^N\bigr|\bigr]\ud s}\biggr)\bigg),
	\end{align*}
	where $C_{L^{\prime},K,T,f} \in \R_+$ is a constant depending on the constant $L'$, the kernel~$K$, $T$ and $f$ through the quantity $\int_{\,U}{\bigl(f(u) \wedge f(u)^2\bigr)\pi(\ud u)}$. Relying finally on Lemma \ref{lem:barX}--\eqref{eq:lem articular_case}, while choosing $\delta=(M \wedge N)^{1/4}$ and $\varepsilon=1/(M \wedge N)^{1/4}$, we get
	\begin{align*}
		\EE\bigl[\bigl|\barX_t^{M,N}\bigr| + \bigl|\Xi_t^{M,N}\bigr|\bigr] &\leq C_{L^{\prime},K,T,f}\int_0^t {\EE\bigl[\bigl|\barX_s^{M,N}\bigr| + \bigl|\Xi_s^{M,N}\bigr|\bigr]\ud s}\\
		&\quad+ C_{L,L^{\prime},K,T,f,X_0}\biggl(\frac{1}{\log(M \wedge N)} + \frac{1}{(M \wedge N)^{1/8}}\biggr),
	\end{align*}
	where $C_{L,L^{\prime},K,T,f,X_0} \in \R_+$ and for which an application of Gr\"onwall's lemma provides the claim.
\end{proof}
\begin{thm}\label{thm:existence} Suppose that $X_0\geq0$, Assumptions \ref{ass:conditions1}, \ref{ass:nonneg} and \ref{ass:global_conditions} hold true, $K \in C^2(\R_+,\R_+)$ is non-increasing, preserves non-negativity and such that $K(0)>0$.
	Then, there exists a non-negative c\`adl\`ag solution $X=(X_t)_{t\in[0,T]}$ of Equation~\eqref{eq:SVE}. Besides, there exists a constant $C_{L,L^{\prime},K,T,f,X_0} \in \R_+$ such that
	\begin{equation*}
		{\red\sup_{t\in[0,T]}\EE\bigl[\bigl|X_t- \barX_t^N\bigr|\bigr]} \leq C_{L,L^{\prime},K,T,f,X_0}\,\frac{1}{\log( N)}.
	\end{equation*}
	$$ $$
\end{thm}
\begin{proof}
	We consider the Banach space of $\FF$-progressively measurable processes $(Y)_{t\in[0,T]}$ such that
	\begin{equation*}
		\quad \lVert Y \rVert := \underset{[0,T]}{\sup \ }\EE\bigl[|Y|\bigr]<\infty.
	\end{equation*}
	Applying now Proposition \ref{prop:barX} for $M=N+1$ and the subsequence $N = \lceil e^{n^2} \rceil$, $n\geq1$, where by a slight abuse of notation, we denote this subsequence  $(\bar{X}^n)_{n\ge 1}$ in the proof while $(\bar{X}^N)_{N\ge 1}$ is the full sequence, we  have 
	\begin{equation*}
		\bigl\lVert \barX^{n+1} - \barX^n \bigr\rVert \leq C_{L,L^{\prime},K,T,f,X_0}\,\frac{1}{n^2}.
	\end{equation*}
	$(\barX^n)_{n\geq1}$ is therefore a Cauchy sequence, thus yielding the existence of $ X$ progressively measurable such that $\lVert \barX^n -  X \rVert \to 0$ as $n\to+\infty$. Besides, we have $\lVert \hatX^n -  X \rVert \to 0$ and $\lVert \xi^n -  X \rVert \to 0$ by Lemma \ref{lem:barX}. Resorting to Lemma \ref{lem:nonnegative_approx}, notably $\PP(\hatX_t^n \geq 0)=1$ for all $t\in[0,T]$ and $n\geq1$, it results that $\PP( X_t\geq0)=1$ for all $t\in[0,T]$. Noting also from Proposition \ref{prop:first_moment_estimates} that there exists a constant $C_{L,K,T,X_0}\in\R_+$ such that $\lVert \hatX^n \rVert \leq C_{L,K,T,X_0}$ for all $n\geq1$, we then get $\lVert  X \rVert \leq C_{L,K,T,X_0}$.
	
	We now introduce the process
	\begin{equation*}
		\tildeX_t := \liminf_{m\to+\infty}\,m\int_{t-\frac{1}{m}}^t {X_s\,\ud s}.
	\end{equation*}
	The process $\tilde{X}$ is predictable and $\tilde{X}_t=X_t$ dt-a.e. almost surely by the Lebesgue differentiation theorem. We have thus $\E[|\tilde{X}_t-X_t|]=0$ dt-a.e. and $\underset{[0,T]}{\text{ess} \sup \ }\EE\bigl[|\barX^n-\tilde{X}|\bigr]\le \|\barX^n-X\|\to 0$.

	We now show that $X$ is, up to a modification, c\`adl\`ag and solves~\eqref{eq:SVE}.  By Proposition~\ref{prop:semi_martingale}, we have
	\begin{align}
		\bar{X}^n_t&=X_0 + K(0) Z^n_t+ K'(0) \int_0^t Z^n_s ds +\int_0^t \left(\int_0^s K''(s-r) Z^n_r \ud r\right) \ud s, \label{Xn_semimg}\\
		\text{ with }  Z^n_t&=\int_0^t {\biggl(\mu(\xi_{s}^n)\,\ud s + \sigma(\xi_{s}^n)\,\ud B_s + \int_{\,U} {\eta(\xi_{s-}^n, u)\,\tildeN (\ud s, \ud u)}\biggr)}. \notag
	\end{align}
	We also introduce the process $Z=(Z_t)_{t\in[0,T]}$ given by
	\begin{equation*}
		Z_t := \int_0^t {\biggl(\mu( \tilde{X}_s)\,\ud s + \sigma( \tilde{X}_s)\,\ud B_s + \int_{\,U} {\eta( \tilde{X}_s, u)\,\tildeN (\ud s, \ud u)}\biggr)}, \ t\in[0,T],
	\end{equation*}
	which is well-defined by using Assumption~\ref{ass:conditions1}, $\| \tilde{X}\|<\infty$ and the predictability of $ \tilde{X}$. We then define the process~$\check{X}$ by
	\begin{equation} \label{eq:def_X}
		\check{X}_t=X_0 + K(0) Z_t+ K'(0) \int_0^t Z_s ds +\int_0^t \left(\int_0^s K''(s-r) Z_r \ud r\right) \ud s.
	\end{equation}
	By construction as a stochastic integral, the process $Z$ is c\`adl\`ag. So is the process $\check{X}$.

	We have by the triangle inequality and It\^o isometry
	\begin{align*}
		\E[|Z^n_t-Z_t|] &\le  \int_0^t  \E[|\mu(\xi^n_s) -\mu( \tilde{X}_s)|] \ud s + 2 \int_0^t  \int_{f(u)\ge 1} \E[|\eta(\xi^n_s,u)-\eta( \tilde{X}_s,u)|]\pi( \ud u) \ud s \\
		&+ \E \left[\int_0^t (\sigma(\xi^n_s) -\sigma( \tilde{X}_s))^2 \ud s \right]^{1/2} + \E \left[ \int_0^t \int_{f(u)<1}(\eta(\xi^n_s,u)-\eta( \tilde{X}_s,u))^2\pi( \ud u) \ud s \right]^{1/2} .
	\end{align*}
	Then, by Assumption~\ref{ass:global_conditions} and using that $\tilde{X}_t=X_t$ dt-a.e. almost surely, we get
	\begin{align*}
		\|Z^n-Z\| \le  &L'  T \|\xi^n- X\| + 2\left(\int_{f(u)\ge 1} f(u) \pi (\ud u) \right)T \|\xi^n-X\|^{1/2}+ \left( L' T  \|\xi^n-X\| \right)^{1/2} \\
		&+\left(\left(\int_{f(u)< 1} f(u)^2 \pi (\ud u) \right) T \|\xi^n-X\|\right)^{1/2} 
	\end{align*}
	We get $\|Z^n-Z\|\underset{n\to \infty}{\to} 0$ by using $\lVert \xi^n -  X \rVert \to 0$.
	We easily deduce then from~\eqref{Xn_semimg} and~\eqref{eq:def_X} that $\|\barX^n-\check{X}\| \to 0$. Thus, we get that $\|X-\check{X}\|=0$, i.e.  that $\check{X}$ is a c\`adl\`ag modification of~$X$. Without loss of generality, we may assume $X=\check{X}$. Therefore, $X_{s-}$ exists and is equal to $\tilde{X}_s$ almost surely, so that
	$$Z_t=\int_0^t {\biggl(\mu( X_s)\,\ud s + \sigma( X_s)\,\ud B_s + \int_{\,U} {\eta( X_{s-}, u)\,\tildeN (\ud s, \ud u)}\biggr)}.$$
	This shows that $X$ solves~\eqref{eq:SVE} by using~\eqref{eq:def_X} and Proposition~\ref{prop:semi_martingale}. 
	
	The last inequality follows from Proposition~\ref{prop:barX} that gives $\|\barX^n-\barX^N \|\le C_{L,L^{\prime},K,T,f,X_0}\,\frac{1}{\log( N\wedge \lceil e^{n^2}\rceil)}$ and letting $n\to \infty$.
\end{proof}

We are now in position to prove Theorem~\ref{thm:main_result}, for which the global Assumption~\ref{ass:global_conditions} is replaced with the local Assumption~\ref{ass:local_conditions}.
\begin{proof}[Proof of Theorem \ref{thm:main_result}]
	For every $m\geq 1$, we define $ \pi_m(x)=-m\vee(x\wedge m)$ the projection of $x$ on $[-m,m]$ and the functions $\mu_m,\,\sigma_m:\R\to\R$ and $\eta_m : \R \times U \to \R$ as follows:
	\begin{equation*}
		\begin{gathered}
			\mu_m(x) := \mu(\pi_m(x)), \qquad \sigma_m(x) := \sigma(\pi_m(x)), \qquad
			\eta_m(x,u) := \eta(\pi_m(x),u),\qquad (x,u) \in \R \times U.
		\end{gathered}
	\end{equation*}
	By construction, $\mu_m,\,\sigma_m$ and $\eta_m$ agree with $\mu,\,\sigma$ and $\eta$ on $[-m,m]$. They satisfy Assumption \ref{ass:global_conditions} since $\pi_m$ is Lipschitz and  $\mu$, $\sigma$ and $\eta$ satisfy Assumption~\ref{ass:local_conditions}. Hence, combining Theorems \ref{thm:existence} and \ref{thm:pathwise_uniqueness}, there exists a pathwise unique non-negative c\`adl\`ag solution $X^m=(X_t^m)_{t\geq0}$ of Equation \eqref{eq:SVE}, where we replaced $\mu,\,\sigma$ and $\eta$ with $\mu_m,\,\sigma_m$ and $\eta_m$. 
	
	As in the proof of, e.g., \cite[Theorem V.12.1]{RW00}, we define $\tau_m := \inf\{t \geq 0 : X_t^{m+1} \geq m\}$, for every $m\geq0$. Since $X^{m+1}$ is c\`adl\`ag and $\FF$-adapted by definition, $\tau_m$ is an $\FF$-stopping time. In particular, we have $\PP(X_{t}^{m+1} \leq m,\forall t\in(0,\tau_m))=1$. Using also that $\mu_{m+1},\,\sigma_{m+1}$ and $\eta_{m+1}$ agree with $\mu_m,\,\sigma_m$ and $\eta_m$ on $[0,m]$, we get that $X^{m+1}$ solve the same stochastic Volterra equation as~$X^m$ up to $\tau_m$. It then holds that $\PP(X_t^m = X_t^{m+1},\forall t\in[0,\tau_m))=1$ by Theorem~\ref{thm:pathwise_uniqueness}. Therefore, we get that $\tau_m = \inf\{t \geq 0 : X_t^m \geq m\}$ almost surely and then $\tau_m\ge \tau_{m-1}$ for $m \geq 1$. Thus, $(\tau_m)_{m\geq0}$ is non-decreasing almost surely.
	
	We now prove that $\tau_m\to \infty$ almost surely. Let $T>0$. Since the coefficients $\mu_m$, $\sigma_m$ and $\eta_m$ satisfy Assumption~\ref{ass:conditions1} (with the same constant $L$ because $|\pi_m(x)|\le |x|$) and we get by Lemma~\ref{lem:moments1} 
	\be\label{intermed_moments}\forall m \ge 1, \forall t \in[0,T], \ \E[X^m_t]\le C_{T,L,K,X_0}.
	\ee
	By Proposition~\ref{prop:semi_martingale}, we also have
	\be\label{Xm_stopped}
	X^m_{T\wedge \tau_m} =X_0+K(0)Z^m_{T\wedge \tau_m}+\int_0^{T\wedge \tau_m} K'(T\wedge \tau_m-t)Z^m_t \ud t,
	\ee
	with $Z^m_t=  \int_0^t  \left(\mu_m(X^m_s)\,\ud s + \sigma_m(X^m_s)\,\ud B_s + \int_{\,U} {\eta_m(X^m_{s-}, u)\,\tildeN (\ud s, \ud u)}\right)$.  By using Proposition~\ref{prop:prem_first_moment_estimate} with $H=1$, $p=0$, $\tau=q=t$ and \eqref{intermed_moments}, we get  
	\begin{equation*} 
		\forall m \ge 1, \forall t \in[0,T], \ \E[|Z^m_t|]\le \tilde{C}_{T,L,K,X_0},
	\end{equation*}
	for some constant $\tilde{C}_{T,L,K,X_0} \in \R_+$. Besides,  
	using Assumption~\ref{ass:conditions1} and $\tau_m = \inf\{t \geq 0 : X_t^m \geq m\}$, we get the martingale property of the stochastic integrals defining $Z_{\cdot\wedge\tau_m}^m$ and then
	$$\E[Z^m_{T\wedge \tau_m}]= \int_0^T  \E[ \ind_{t< \tau_m} \mu_m(X^m_t)] \,\ud t  \le T L(1+C_{T,L,K,X_0}),$$
	by using~\eqref{intermed_moments}. 
	From~\eqref{Xm_stopped}, we get
	$$\E[X^m_{T\wedge \tau_m}]\le X_0+K(0)T L(1+C_{T,L,K,X_0})+T \tilde{C}_{T,L,K,X_0} \left(\max_{[0,T]} |K'|\right). $$
	This bound does not depend on~$m$ and we have $\E[X^m_{T\wedge \tau_m}]\ge m\,\PP(\tau_m <T)$ as $X_{\tau_m}^m \geq m$ almost surely since $X^m$ is c\`adl\`ag. This shows that $\PP(\tau_m<T)\to 0$ and then that $\tau_m\to \infty$ almost surely since $\tau_m$ is a non-decreasing sequence. 
	
	Finally, we define the process $X=(X_t)_{t\geq0}$ by $X_t=X^m_t$ on $\{t<\tau_m\}$. This is well defined since the processes $X^{m+p}$ and $X^m$ coincide on $\{t<\tau_m\}$. $X$ is thus a c\`adl\`ag solution of Equation~\eqref{eq:SVE} up to $\tau_m$, for any $m\ge 1$, which gives that $X$ solves Equation~\eqref{eq:SVE} for all $t\ge 0$. Last, it is pathwise unique by Theorem \ref{thm:pathwise_uniqueness}, ensuring the final claim.
\end{proof}

\section{Applications: L\'evy-driven stochastic Volterra equations}\label{sec:sec6}
We investigate in this section the following one-dimensional L\'evy-driven stochastic Volterra equation of convolution type:
\be\label{eq:SVE_Levy}
X_t = X_0 + \int_0^t {K(t-s)\,\mu(X_{s})\,\ud s} + \int_0^t {K(t-s)\,\sigma(X_{s})\,\ud B_s} + \int_0^t {K(t-s)\,\gamma(X_{s-})\,\ud L_s}, 
\ee
where $X_0\in\R$, on the filtered probability space $(\Omega,\cF,\FF,\PP)$ described in Section \ref{sec:sec2} and supporting the following independent random elements:
\begin{itemize}
	\item an $\FF$-Brownian motion $B=(B_t)_{t\geq0}$;
	\item an $\FF$-L\'evy process $L=(L_t)_{t\geq0}$ with triplet $(0,0,\nu)$ where $\nu$ is the L\'evy measure on $\R_+$: $$\nu(\ud u) := u^{-1-\alpha}\,\ind_{\{u>0\}}\,\ud u, \quad \text{ with }\alpha\in(1,2),$$ which means that $L$ is a spectrally positive compensated $\alpha$-stable L\'evy process. Note that $\alpha$ is chosen such that $\int_0^{+\infty} {\bigl(u\wedge u^2\bigr)\nu(\ud u)} = \frac{1}{2-\alpha}+ \frac{1}{\alpha-1}<\infty$.
\end{itemize}
We consider the following ingredients:
\begin{itemize}
	\item $\mu, \sigma, \gamma : \R \to \R$ are continuous functions;
	\item $K : \R_+ \to \R_+$  is a non-negative continuous function.
\end{itemize}
\begin{ass}\label{ass:Levy_conditions}
	Suppose that $\sigma(0)=\gamma(0)=0$, $\mu(0)\geq0$, $x\mapsto\gamma(x)$ is non-decreasing and 
	\begin{enumerate}
		\item[(i)] there exists a constant $L > 0$ such that 
		\begin{equation*}
			\bigl|\mu(x)\bigr| + \sigma(x)^2 + \bigl|\gamma(x)\bigr|^{\alpha} \leq L\bigl(1 + |x|\bigr), \qquad \text{for all } x\in\R;
		\end{equation*}
		\item[(ii)] for every $m\geq1$, there exists a constant $L_m^{\prime} > 0$ such that
		\begin{equation*}
			\bigl|\mu(x) - \mu(y)\bigr| + \bigl|\sigma(x) - \sigma(y)\bigr|^2 + \bigl|\gamma(x) - \gamma(y)\bigr|^2 \leq L_m^{\prime}\,\bigl|x - y\bigr|, \qquad \text{for all } (x,y)\in[-m,m]^2.
		\end{equation*}
	\end{enumerate}
\end{ass}
\begin{thm}\label{thm:Levy_result}
	Suppose that $X_0\geq0$, Assumption \ref{ass:Levy_conditions} holds true and $K \in C^2(\R_+,\R_+)$ is non-increasing, preserves non-negativity and $K(0)>0$. Then, there exists a pathwise unique non-negative c\`adl\`ag solution $X=(X_t)_{t\geq0}$ of Equation~\eqref{eq:SVE_Levy}.
\end{thm}
\begin{proof}
	We express the L\'evy process $L$ by means of its L\'evy--It\^o decomposition,
	\begin{equation*}
		L_t = \int_0^t \int_0^{+\infty} {u\,\tildeN (\ud s, \ud u)},
	\end{equation*}
	almost surely for all $t\geq0$, where $N$ is the $\FF$-Poisson random measure on $[0,+\infty)^2$ representing the jumps of $L$ (see, e.g., \cite[Example II.4.1]{IW89}). In this, \eqref{eq:SVE_Levy} reduces to a special case of Equation~\eqref{eq:SVE} with $$U=\R_+,\ \pi(\ud u)=\nu(\ud u) \text{ and }\eta(x,u) = u\,\gamma(x).$$ In particular, it holds that $\eta(0,u)=0$ and $x\mapsto\eta(x,u)$ is non-decreasing for every $u\geq0$ under Assumption \ref{ass:Levy_conditions}.   By a change of variable, we also have that for all $x\in \R$,
	\begin{align*}
		\int_0^{+\infty} {\bigl(\bigl|u \gamma(x)\bigr|\wedge (\gamma(x)u)^2\bigr) u^{-1-\alpha} \ud u}= |\gamma(x)|^\alpha \int_0^{+\infty} {\bigl(u \wedge u^2\bigr)\nu(\ud u)}.
	\end{align*}
	Therefore, Assumption \ref{ass:conditions1} holds true under Assumption \ref{ass:Levy_conditions}-(i), and it can be easily verified that Assumption \ref{ass:local_conditions} holds true as well under Assumption \ref{ass:Levy_conditions}-(ii) since {\red $|\eta(x,u) - \eta(y,u)| = u|\gamma(x) - \gamma(y)|$ and} $\int _0^\infty (u\wedge u^2)\nu(du)<\infty$. The claim thus follows from a direct application of Theorem~\ref{thm:main_result}.
\end{proof}

Consider now, as a special case of Equation~\eqref{eq:SVE_Levy}, the  L\'evy-driven stochastic Volterra equation
\be\label{eq:SVE_Levy_bis}
\begin{aligned}
	X_t &= X_0 + \int_0^t {K(t-s)\bigl(a - \kappa\,X_{s-}\bigr)\ud s} + \sigma\int_0^t {K(t-s)\,\bigl|X_{s-}\bigr|^{1/2}\ud B_s}\\
	&\quad+ \eta\int_0^t {K(t-s)\,\sgn(X_{s-})\,\bigl|X_{s-}\bigr|^{1/\alpha}\ud L_s},  
\end{aligned}
\ee
where $X_0\in\R$, $\kappa\in\R$, $a, \sigma, \eta \geq0$, $\alpha\in(1,2)$ and $(B,L)$ defined as above. We also consider $K$ completely monotone as a special case of non-increasing non-negativity preserving $C^2$ kernel. 
\begin{cor}\label{cor:Levy_result}
	Suppose that $X_0\geq0$ and $K$ is completely monotone such that $0<K(0)<+\infty$. Then, there exists a pathwise unique non-negative c\`adl\`ag solution $X=(X_t)_{t\geq0}$ of Equation~\eqref{eq:SVE_Levy_bis}.
\end{cor}
\begin{proof}
	It suffices to verify that the functions $x\mapsto a-\kappa\,x$, $x\mapsto\sigma|x|^{1/2}$ and $x\mapsto\eta\sgn(x)|x|^{1/\alpha}$, for $x\in\R$, satisfy Assumption \ref{ass:Levy_conditions}. The presence of $\sgn(\cdot)$ ensures that $x\mapsto\eta\sgn(x)|x|^{1/\alpha}$ is non-decreasing. We also observe that Assumption \ref{ass:Levy_conditions}-(i) is directly satisfied. The validity of Assumption \ref{ass:Levy_conditions}-(ii) then follows from the H\"older condition of $x \mapsto x^{1/2}$ and $x \mapsto x^{1/\alpha}$ on $\R_+$.
\end{proof}
Under the conditions of Corollary \ref{cor:Levy_result}, it holds that the pathwise unique c\`adl\`ag solution $X=(X_t)_{t\geq0}$ of Equation \eqref{eq:SVE_Levy_bis} is non-negative. We can thus rewrite it as
\be\label{eq:Volterra_alphastable}
	X_t = X_0 + \int_0^t {K(t-s)\bigl(a - \kappa\,X_{s}\bigr)\ud s} + \sigma\int_0^t {K(t-s)\sqrt{X_{s}}\,\ud B_s} + \eta\int_0^t {K(t-s)\sqrt[\alpha]{X_{s-}}\,\ud L_s}. 
\ee
It corresponds to a Volterra extension of the so-called \emph{$\alpha$-stable Cox--Ingersoll--Ross process}, refer e.g. to \cite{LM15,JMS17,JMSZ21} and \cite[Section 2.6.2]{phdthesis_szulda} for further information. Let us note that this process can be seen as a Volterra affine process for which the calculation of Laplace transform can be made semi-explicit. For $T>0$, $u\in \R_-$ and an integrable nonpositive function $f:[0,T]\to \R_-$, $\exp \left(u X_T+ \int_0^T f(T-s) X_s \ud s \right) \le 1$ is integrable, and  we can formally follow the steps of~\cite[Theorem 4.3]{AJLP19} (recalling that $\E[e^{uL_t}]=\exp\left(\frac{t|u|^\alpha}{\cos \left(\frac{\pi}{2}(2-\alpha)\right)} \right)$) to get
$$\E \left[\exp \left(u X_T+ \int_0^T f(T-s) X_s \ud s \right) \bigg|\cF_t\right]= \exp(Y_t), $$
where 
$$\begin{cases}
Y_t=Y_0 +\int_0^t \psi(T-s) \sigma \sqrt{X_s} \ud B_s +\int_0^t \psi(T-s) \eta \sqrt[\alpha]{X_{s-}}\,\ud L_s \\
\phantom{Y_t=Y_0}-\int_0^t X_s \left( \frac{\sigma^2}{2} \psi^2(T-s) + \frac{\eta^\alpha}{\cos \left(\frac{\pi}{2}(2-\alpha)\right)}|\psi(T-s)|^\alpha \right) \ud s,\\
Y_0=u X_0 +X_0 \int_0^T \left(f(s) -\kappa \psi(s) +\frac{\sigma^2}{2} \psi^2(s) + \frac{\eta^\alpha}{\cos \left(\frac{\pi}{2}(2-\alpha)\right)}|\psi(s)|^\alpha\right) \ud s +a \int_0^T \psi(s) \ud s,
\end{cases}$$
and $\psi$ is the solution of the Volterra equation
$$\psi(t)=uK(t)+\int_0^t K(t-s)\left(f(s) -\kappa \psi(s) +\frac{\sigma^2}{2} \psi^2(s) + \frac{\eta^\alpha}{\cos \left(\frac{\pi}{2}(2-\alpha)\right)}|\psi(s)|^\alpha\right) \ud s,\ t \in[0,T].  $$
The characteristic function of Volterra affine processes with jumps has been very recently studied by Abi Jaber~\cite{AJ21} and Bondi et al.~\cite{BLP24} under the assumption of square integrable jumps, which is not satisfied by~\eqref{eq:Volterra_alphastable}. However, their analysis and in particular the one of~\cite[Theorem 2.5]{AJ21} could be useful to get that $\psi$ is well defined. A careful study requires further developments and is beyond the scope of this paper. 

{\red\begin{rem}
	\cite{JMSZ21} have recently proposed an extension of the Heston model with an alpha-stable Cox--Ingersoll--Ross process for the volatility. In particular, they show the effect of the parameter $\alpha$ on the volatility smile. On the other hand, \cite{EER19} have introduced the rough Heston model that has attracted a great interest. Besides, \cite{AJEE19a} have shown that multi-factor approximations of the fractional kernel can produce very similar smiles. Therefore, the solution of Equation~\eqref{eq:Volterra_alphastable} represents a natural candidate for the volatility process that extends both the alpha-Heston and multi-factor Heston models while preserving the affine structure.
\end{rem}}

\begin{appendix}
	
	\section{Auxiliary results}\label{app:sec5}
	In this appendix, we present some auxiliary results for processes obtained as the integration of a kernel with respect to a semi-martingale. Namely, we consider the following objects:
	\begin{itemize}
		\item $\xi=(\xi_t)_{t\geq0}$ is an $\FF$-adapted c\`adl\`ag process;
		\item $H:\R_+^2\to\R$ is a Borel function such that $H(t,s)=0$ whenever $0 \leq t < s$ and for all $T>0$,\\ $\|H\|_{T} :=\sup_{0 \leq s,t \leq T}|{\red H(t,s)}| <\infty$.
	\end{itemize}
	Most of the time, we will work with ${\red H(t,s)}=K(t-s)$, but sometimes it will be convenient to work with more general kernels. 
	For $0 \leq p \leq q \leq t$, we define
	\begin{equation}\label{def_XI}
		\cX(p,q,t) := \int_p^q {H(t,s)\biggl(\mu(\xi_s)\,\ud s + \sigma(\xi_s)\,\ud B_s + \int_{\,U} {\eta(\xi_{s-}, u)\,\tildeN (\ud s, \ud u)}\biggr)},
	\end{equation}
	where $B$, $\tilde{N}$ and $(\mu,\sigma,\eta)$ were introduced in Section~\ref{sec:sec2} and satisfy Assumption \ref{ass:conditions1}. 
	\begin{lem}\label{lem:wellposedness}
		Under Assumption \ref{ass:conditions1}, $\cX(p,q,t)$ is well defined and almost surely finite.
	\end{lem}
	\begin{proof}
		It amounts to checking whether 
		\begin{multline}
			\int_p^q {\bigl|H(t,s)\bigr|\biggl(\bigl|\mu(\xi_s)\bigr| + \int_{\{|\eta(\xi_{s},u)|\geq1\}} {\!\bigl|\eta(\xi_{s}, u)\bigr|\,\pi(\ud u)}\biggr)\ud s}\\
			+ \int_p^q {H(t,s)^2\biggl(\sigma(\xi_s)^2 + \int_{\{|\eta(\xi_{s},u)|<1\}} {\!\eta(\xi_{s}, u)^2\,\pi(\ud u)}\biggr)\ud s} < +\infty. \label{condition_lemme}
		\end{multline}
		Indeed, let us define the sequence of stopping times for $n\ge 1$:
		\begin{multline*}
			\tau_n=\inf \Bigg\{ r\in[p,q] :\int_p^r {\bigl|H(t,s)\bigr|\biggl(\bigl|\mu(\xi_s)\bigr| + \int_{\{|\eta(\xi_{s},u)|\geq1\}} {\!\bigl|\eta(\xi_{s}, u)\bigr|\,\pi(\ud u)}\biggr)\ud s}\\
			+ \int_p^r {H(t,s)^2\biggl(\sigma(\xi_s)^2 + \int_{\{|\eta(\xi_{s},u)|<1\}} {\!\eta(\xi_{s}, u)^2\,\pi(\ud u)}\biggr)\ud s} \geq n \Bigg\}.
		\end{multline*}
		By classical results (see e.g.~\cite{IW89}), $\cX(p,q\wedge \tau_n,t)$ is well defined and since $\tau_n\ge q$ almost surely for $n$ large enough, $\cX(p,q,t)$ is also well defined.

		To prove~\eqref{condition_lemme}, we use Assumption \ref{ass:conditions1} and write
		\begin{multline*}
			\int_p^q {\Bigl(\bigl|H(t,s)\bigr| \vee H(t,s)^2\Bigr)\biggl(\bigl|\mu(\xi_s)\bigr| + \sigma(\xi_s)^2 + \int_{\,U} {\!\bigl(\bigl|\eta(\xi_s, u)\bigr|\wedge\eta(\xi_s, u)^2\bigr)\pi(\ud u)}\biggr)\ud s}\\
			\leq L\Bigl(\lVert H \rVert_t \vee \lVert H \rVert_t^2\Bigr)\biggl(t + \int_0^t {\xi_s\,\ud s}\biggr),
		\end{multline*}
		{\red where $\|H\|_t$ has been defined above. The right-hand side of the last equation is then almost surely finite since $\xi$ is a c\`adl\`ag process.}
	\end{proof}
	
	\begin{prop}\label{prop:prem_first_moment_estimate}
		Let $\tau$ be an $\FF$-stopping time such that $p \leq \tau \leq q \leq t$ almost surely. 
		Under Assumption \ref{ass:conditions1}, there exists a constant $C_{L}\in\R_+$ such that 
		\begin{equation*}
			\EE\bigl[\bigl|\cX(p,\tau,t)\bigr|\,\bigl|\cF_p\bigr] \leq C_{L} \|H\|_t \,\Biggl(q-p + \int_p^q {\EE\bigl[\mathbf{1}_{s<\tau }  |\xi_{s }|\,\bigl|\cF_p\bigr]\ud s} + \biggl(q-p + \int_p^q {\EE\bigl[\mathbf{1}_{s<\tau }  |\xi_{s}|\,\bigl|\cF_p\bigr]\ud s}\biggr)^{1/2}\Biggr),
		\end{equation*}
		where left and right hand sides may be infinite. 
	\end{prop}
	\begin{proof}
		From~\eqref{def_XI}, we use the triangle inequality and take the conditional expectation to get
		\begin{align*}
			\EE\bigl[\bigl|\cX(p,\tau,t)\bigr|\,\bigl|\cF_p\bigr] &\leq \int_p^\tau {\bigl|H(t,s)\bigr|\,\EE\bigl[\bigl|\mu(\xi_s)\bigr|\,\bigl|\cF_p\bigr]\ud s} + \EE\biggl[\biggl|\int_p^\tau {H(t,s)\,\sigma(\xi_s)\,\ud B_s}\biggr|\,\biggl|\cF_p\biggr]\\  &\quad+ \EE\biggl[\biggl|\int_p^\tau\ \int_{\{|\eta(\xi_{s-},u)|<1\}} {\!H(t,s)\,\eta(\xi_{s-}, u)\,\tildeN (\ud s, \ud u)}\biggr|\,\biggl|\cF_p\biggr]\\
			&\quad+ \EE\biggl[\biggl|\int_p^\tau\int_{\{|\eta(\xi_{s-},,u)|\geq1\}} {\!H(t,s)\,\eta(\xi_{s-}, u)\,\tildeN (\ud s, \ud u)}\biggr|\,\biggl|\cF_p\biggr],
		\end{align*}
		where we also split the integral between small and large jumps. By Assumption~\ref{ass:conditions1}, the second and third terms of the right-hand side can be upper bounded as follows:
		\begin{align*}
			&\EE\biggl[\biggl|\int_p^q {\mathbf{1}_{s<\tau } H(t,s)\,\sigma(\xi_s)\,\ud B_s}\biggr|\,\biggr|\cF_p\biggr] + \EE\biggl[\biggl|\int_p^q\int_{\{|\eta(\xi_{s-},u)|<1\}} {\! \mathbf{1}_{s<\tau } H(t,s)\,\eta(\xi_{s-}, u)\,\tildeN (\ud s, \ud u)}\biggr|\,\biggr|\cF_p\biggr]\\
			&\quad\leq \sqrt{2}\,\EE\Biggl[\biggl(\int_p^q {\mathbf{1}_{s<\tau } H(t,s)\,\sigma(\xi_s)\,\ud B_s}\biggr)^2 + \biggl(\int_p^q\int_{\{|\eta(\xi_{s-},u)|<1\}} {\! \mathbf{1}_{s<\tau } H(t,s)\,\eta(\xi_{s-}, u)\,\tildeN (\ud s, \ud u)}\biggr)^2\,\Biggr|\cF_p\Biggr]^{1/2}\\
			&\quad = \sqrt{2}\,\EE\biggl[\int_p^q {\mathbf{1}_{s<\tau } H(t,s)^2\biggl(\sigma(\xi_s)^2 + \int_{\{|\eta(\xi_s,u)|<1\}} {\!\eta(\xi_s, u)^2\,\pi(\ud u)}\biggr)\ud s}\,\biggr|\cF_p\biggr]^{1/2}\\
			&\quad\quad\leq \sqrt{2\,L}\,\lVert H \rVert_t\biggl(q-p + \int_p^q {\EE\bigl[\mathbf{1}_{s<\tau }  |\xi_{s }|\bigr|\cF_p\bigr]\ud s}\biggr)^{1/2},
		\end{align*}
		where we have used Cauchy--Schwarz and Jensen in a row for the first inequality, It\^o isometry for the equality, and Assumption \ref{ass:conditions1} for the last inequality. 
		The first term is simply upper bounded by
		$$\int_p^q { \mathbf{1}_{s<\tau } \bigl|H(t,s)\bigr|\,\EE\bigl[\bigl|\mu(\xi_s)\bigr|\,\big|\cF_p\bigr]\ud s}\le L \|H\|_t \biggl(q-p + \int_p^q {\EE\bigl[ \mathbf{1}_{s<\tau }   |\xi_{s }|\big|\cF_p\bigr]\ud s}\biggr).$$
		We then use localization for the fourth term and introduce $\tau_n=\inf\{t\ge p : \int_p^t |\xi_s| \ud s \ge n \},$ and have $\tau_n \to +\infty$ a.s. since $\xi$ has c\`adl\`ag paths. We write
		\begin{align*}
			&\EE\biggl[\biggl|\int_p^{\tau \wedge \tau_n} \int_{\{|\eta(\xi_s,u)|\geq1\}} {H(t,s)\,\eta(\xi_s, u)\,N (\ud s, \ud u)} - \int_p^{\tau \wedge \tau_n} \int_{\{|\eta(\xi_s,u)|\geq1\}} {H(t,s)\,\eta(\xi_s, u)\,\ud s\,\pi(\ud u)}\biggr|\,\bigg|\cF_p\biggr]\\
			&\quad\leq \EE\biggl[\int_p^q \mathbf{1}_{s<\tau \wedge \tau_n} \int_{\{|\eta(\xi_s,u)|\geq1\}} {\bigl|H(t,s)\bigr|\,\bigl|\eta(\xi_s, u)\bigr|\,N (\ud s, \ud u)} \\
			& \quad \quad \quad  + \int_p^q \mathbf{1}_{s<\tau \wedge \tau_n} \int_{\{|\eta(\xi_s,u)|\geq1\}} {\bigl|H(t,s)\bigr|\,\bigl|\eta(\xi_s, u)\bigr|\,\pi(\ud u)\,\ud s}\,\bigg|\cF_p\biggr]\\
			&\quad\quad\leq 2\,L\, \Vert H \Vert_t\biggl(q-p + \int_p^q {\EE\bigl[\mathbf{1}_{s<\tau }  |\xi_{s }|\big| \cF_p\bigr]\ud s}\biggr),
		\end{align*}
		for which we in particular used $\EE\bigl[\int_p^q \mathbf{1}_{s<\tau \wedge \tau_n}  \int_{\{|\eta(\xi_s,u)|\geq 1 \}} {|H(t,s)\eta(\xi_s, u)|\,\tildeN (\ud s, \ud u)}\,|\cF_p\bigr]=0$ (see, e.g., \cite[Section II.3]{IW89}) and Assumption~\ref{ass:conditions1} for the last inequality. We then apply Fatou's Lemma and finally get
		\begin{multline*}
			\EE\bigl[\bigl|\cX(p,\tau,t)\bigr|\,\big|\cF_p\bigr] \leq  3\,L\,\Vert H \Vert_t\biggl(q-p + \int_p^q {\EE\bigl[\mathbf{1}_{s<\tau }   |\xi_{s }|\big|\cF_p\bigr]\ud s}\biggr)\\
			+ \sqrt{2\,L}\,\Vert H \Vert_t\biggl(q-p + \int_p^q {\EE\bigl[\mathbf{1}_{s<\tau }  |\xi_{s }|\big|\cF_p\bigr]\ud s}\biggr)^{1/2},
		\end{multline*}
		which yields the claim. 
	\end{proof}
	
	\begin{prop}\label{prop:semi_martingale}
		Suppose that $H(t,s)=K(t-s)$ for $s\le t$ with $K \in C^2(\R_+,\R)$. Let us consider  the process $\cX_t=\cX(0,t,t)$ for $t\ge 0$.
		Then, under Assumption \ref{ass:conditions1}, $\cX=(\cX_t)_{t\geq0}$  is an $\FF$-semimartingale  and satisfies
		\begin{align*}
			\cX_t=&\int_0^t K(t-s) \biggl(\mu(\xi_s)\,\ud s + \sigma(\xi_s)\,\ud B_s + \int_{\,U} {\eta(\xi_{s-}, u)\,\tildeN (\ud s, \ud u)}\biggr) \\
			=& K(0){\red Y_t} + \int_0^t K'(t-s){\red Y_s} \ud s \\
			=& K(0){\red Y_t} + K'(0) \int_0^t {\red Y_s} \ud s + \int_0^t \left( \int_0^s K''(s-r) {\red Y_{r}} \ud r \right) \ud s,
		\end{align*}
		with ${\red Y_t}=\int_0^t  \left(\mu(\xi_s)\,\ud s + \sigma(\xi_s)\,\ud B_s + \int_{\,U} {\eta(\xi_{s-}, u)\,\tildeN (\ud s, \ud u)}\right)$ for $t\ge 0$.
	\end{prop}
	
	\begin{proof}  We write $K(t-s)=K(0)+K(t-s)-K(0)$, so that $\cX_t=K(0){\red Y_t}+\int_0^t(K(t-s)-K(0)) {\red\ud Y_s}$. We apply It\^o's formula to $(K(t-s)-K(0)){\red Y_s}$ between $s=0$ and $s=t$ and get
		$$0= \int_0^t -K'(t-s) {\red Y_s}\ud s +\int_0^t (K(t-s)-K(0))  {\red\ud Y_s},$$
		leading to the first claim.
		
		Then, we have $\int_0^t K'(t-s) {\red Y_s} \ud s =K'(0) \int_0^t {\red Y_s} \ud s+ \int_0^t (K'(t-s)-K'(0)) {\red Y_s} \ud s $. Using Fubini's theorem, we get
		$$ \int_0^t (K'(t-r)-K'(0)) {\red Y_r} \ud r = \int_0^t \left( \int_r^t K''(s-r) \ud s \right){\red Y_{r}} \ud r =\int_0^t \left(\int_0^{s} K''(s-r)  {\red Y_{r}} \ud r \right)\ud s.$$
	\end{proof}
	\begin{rem}
		To get the semimartingale property, it is enough to assume that $K \in C^1(\R_+,\R)$ by using a stochastic Fubini argument as done by PrÃ¶mel and Scheffel~\cite[Lemma 3.6]{PS23} with  the help of \cite[Proposition A.2]{BDMKR97} to handle the Poisson stochastic integral. This leads to 
		$\cX_t=K(0){\red Y_t}+\int_0^t \left(\int_0^s K'(s-r) {\red\ud Y_r}\right) \ud s$, at the price of more involved arguments. 
		However, using this representation, it is not clear then how to bound the term ``$\uV$'' appearing in the proofs of Theorem~\ref{thm:pathwise_uniqueness} and Proposition~\ref{prop:barX} without assuming that $K \in C^2(\R_+,\R)$. Since we need anyway a $C^2$ kernel for our main results, we decided to state Proposition~\ref{prop:semi_martingale} this way, since it uses very simple arguments. 
	\end{rem}
	
	\section{On a variant of the Yamada--Watanabe approximation technique}\label{sec:appendix}
	We discuss below a variant of the Yamada--Watanabe approximation technique \cite{YW71}, especially used by \cite{Yamada78,Alfonsi05,GR11,LT19a,LT19b} to derive strong rates of convergence, and carried out in Section \ref{sec:sec4} to prove the existence of a strong solution to~\eqref{eq:SVE}. For $\varepsilon\in(0,1)$ and $\delta\in(1,+\infty)$, let $\psi_{\delta,\varepsilon}:\R\to\R_+$ be a non-negative continuous function whose support belongs to $[\varepsilon/\delta, \varepsilon]$ and such that
	\begin{equation*}
		\int_{\varepsilon/\delta}^{\varepsilon} {\psi_{\delta,\varepsilon}(x)\,\ud x} = 1 \quad \text{and} \quad 0\leq\psi_{\delta,\varepsilon}(x)\leq\frac{2}{x\log\delta}\,\ind_{[\varepsilon/\delta, \varepsilon]}(x), \quad \text{for all }x\in\R.
	\end{equation*}
	We then approximate the absolute value by the functions $\varphi_{\delta,\varepsilon} \in C^2(\R,\R_+)$, for $\varepsilon\in(0,1)$ and $\delta\in(1,+\infty)$, defined as
	\begin{equation*}
		\varphi_{\delta,\varepsilon}(x) := \int_0^{|x|}{\biggl(\int_0^y {\psi_{\delta,\varepsilon}(z)\,\ud z}\biggr)\ud y},
	\end{equation*}
	for all $x\in\R$, for which it can be easily checked that 
	\be\label{eq:varphi_bis}
	|x| \leq \varepsilon + \varphi_{\delta,\varepsilon}(x), \quad 0 \leq |\varphi_{\delta,\varepsilon}^{\prime}(x)| \leq 1 \quad \text{and} \quad \varphi_{\delta,\varepsilon}^{\prime\prime}(x) = \psi_{\delta,\varepsilon}(|x|) \leq \frac{2}{|x|\log\delta}\,\ind_{[\varepsilon/\delta, \varepsilon]}(|x|), 
	\ee
	for all $x\in\R$. As in \cite[Section 1.2]{LT19a}, we provide here two lemmata that permit to handle the residual term $\uIV$ arising from the application of It\^o's formula in the proof of Proposition \ref{prop:barX}. In this perspective, we will use Assumption~\ref{ass:global_conditions}-(ii) that we recall for reader's convenience: the function $x\mapsto\eta(x,u)$ is non-decreasing for every $u \in U$ and there exists a non-negative Borel function $f:U\to\R_+$ such that 
	\begin{equation*}
		\bigl|\eta(x, u) - \eta(y, u)\bigr| \leq \bigl|x - y\bigr|^{1/2}f(u), \qquad \text{for all } (x,y,u) \in \R^2 \times U,
	\end{equation*}
	where $f$ satisfies $\int_{\,U} {(f(u) \wedge f(u)^2)\,\pi(\ud u)} < +\infty$.
	
	\begin{lem}\label{lem:lem_varphi_1}
		Let Assumption~\ref{ass:global_conditions}-(ii) hold. For all $(x,y,u) \in \R^2 \times U$, $z:=x-y$ and $c > 0$, it holds that
		\begin{equation*}
			0\le \varphi_{\delta,\varepsilon}\bigl(z + c\,h(x,y,u)\bigr) - \varphi_{\delta,\varepsilon}\bigl(z\bigr) - c\,h(x,y,u)\,\varphi_{\delta,\varepsilon}^{\prime}\bigl(z\bigr) \leq 2\,\bigl(c \vee c^2\bigr)\bigl(f(u) \wedge f(u)^2\bigr)\Bigl(\varepsilon^{1/2}+\frac{1}{\log\delta}\Bigr).
		\end{equation*}
	\end{lem}
	\begin{proof}
		We treat the cases $\{f(u)<1\}$ and $\{f(u)\geq1\}$ separately. Since  $\varphi_{\delta,\varepsilon} \in C^2(\R,\R_+)$, we can apply Taylor's expansion with integral remainder at order two,
		\begin{equation*}
			\varphi_{\delta,\varepsilon}(z + c\,h(x,y,u)) - \varphi_{\delta,\varepsilon}(z) - c\,h(x,y,u)\,\varphi_{\delta,\varepsilon}^{\prime}(z) = c^2\,h(x,y,u)^2\int_0^1 {(1-r)\,\varphi_{\delta,\varepsilon}^{\prime\prime}(z+r\,c\,h(x,y,u))\,\ud r},
		\end{equation*}
		which, since $\varphi_{\delta,\varepsilon}^{\prime\prime}\geq0$ by \eqref{eq:varphi_bis}, implies that the left-hand side is non-negative. Using again \eqref{eq:varphi_bis}, we have
		\begin{equation*}
			\varphi_{\delta,\varepsilon}(z + c\,h(x,y,u)) - \varphi_{\delta,\varepsilon}(z) - c\,h(x,y,u)\,\varphi_{\delta,\varepsilon}^{\prime}(z) \leq \frac{2\,c^2}{\log\delta}\,h(x,y,u)^2\int_0^1 {\frac{\ind_{[\varepsilon/\delta, \varepsilon]}\bigl(|z + r\,c\,h(x,y,u)|\bigr)}{|z + r\,c\,h(x,y,u)|}\,\ud r}.
		\end{equation*}
		Since $x\mapsto\eta(x,u)$ is non-decreasing for every $u \in U$ by Assumption~\ref{ass:global_conditions}-(ii), we have $z\,h(x,y,u) \geq 0$, in particular $|z + r\,c\,h(x,y,u)| \geq |z|$ for all $(x,y,u) \in \R^2 \times U$. Observing also that $\ind_{[\varepsilon/\delta, \varepsilon]}(|z + r\,c\,h(x,y,u)|) \leq \ind_{(0, \varepsilon]}\bigl(|z + r\,c\,h(x,y,u)|\bigr)$, we then get $\ind_{[\varepsilon/\delta, \varepsilon]}(|z + r\,c\,h(x,y,u)|)\leq\ind_{(0, \varepsilon]}(|z|)$. By Assumption~\ref{ass:global_conditions}-(ii), we now use $|h(x,y,u)| \leq |z|^{1/2} f(u)$ in the above inequality and obtain
		\begin{equation*}
			\varphi_{\delta,\varepsilon}(z + c\,h(x,y,u)) - \varphi_{\delta,\varepsilon}(z) - c\,h(x,y,u)\,\varphi_{\delta,\varepsilon}^{\prime}(z) \leq \frac{2\,c^2}{\log\delta}\,\ind_{(0, \varepsilon]}(|z|)\,f(u)^2,
		\end{equation*}
		which gives the result for the case $\{f(u)<1\}$. It also gives that 
		$\varphi_{\delta,\varepsilon}(z + c\,h(x,y,u)) - \varphi_{\delta,\varepsilon}(z) - c\,h(x,y,u)\,\varphi_{\delta,\varepsilon}^{\prime}(z) =0$ for $|z|>\varepsilon$. Therefore, by applying the triangle inequality along with the mean-value theorem under \eqref{eq:varphi_bis}, noticing as well that $\sup_{\,\R}|\varphi_{\delta,\varepsilon}^{\prime}|\leq1$, we have
		\begin{align*}
			\bigl|\varphi_{\delta,\varepsilon}(z + c\,h(x,y,u))& - \varphi_{\delta,\varepsilon}(z) - c\,h(x,y,u)\,\varphi_{\delta,\varepsilon}^{\prime}(z)\bigr|\ind_{(0, \varepsilon]}(|z|) \\&
			\leq \bigl|\varphi_{\delta,\varepsilon}(z + c\,h(x,y,u)) - \varphi_{\delta,\varepsilon}(z)\bigr|\ind_{(0, \varepsilon]}(|z|) + c\,\bigl|h(x,y,u)\bigr|\ind_{(0, \varepsilon]}(|z|)\\
			&\leq 2\,c\,\bigl|h(x,y,u)\bigr|\,\ind_{(0, \varepsilon]}(|z|) \leq 2\,c\,\varepsilon^{1/2}f(u),
		\end{align*}
		where we have as before injected $|h(x,y,u)| \leq |z|^{1/2} f(u)$ at the last step.  This gives the result for the case $\{f(u)\ge 1\}$.
	\end{proof}
	
	The next lemma plays an important role in the proof of Proposition \ref{prop:barX} to analyse the distance between two approximating schemes. It has similarities but is different from~\cite[Lemma 1.4]{LT19a}.
	\begin{lem}\label{lem:lem_varphi_2}
		Let Assumption~\ref{ass:global_conditions}-(ii) hold. For all $(x,y,\alpha,\beta,u) \in \R^4 \times U$, $z:=x-y$ and $c > 0$, it holds that
		\begin{multline*}
			\varphi_{\delta,\varepsilon}\bigl(z + c\,h(\alpha,\beta,u)\bigr) - \varphi_{\delta,\varepsilon}\bigl(z + c\,h(x,y,u)\bigr) - c\,\bigl(h(\alpha,\beta,u)-h(x,y,u)\bigr)\,\varphi_{\delta,\varepsilon}^{\prime}(z)\\
			\leq 6\,\bigl(c \vee c^2\bigr)\,\bigl(f(u) \wedge f(u)^2\bigr)\Bigl(|x-\alpha|^{1/2} + |y-\beta|^{1/2} + \frac{1}{\log\delta} + \frac{\delta}{\varepsilon\log\delta}\bigl(|x-\alpha| + |y-\beta|\bigr)\Bigr).
		\end{multline*}
	\end{lem}
	
	\begin{proof}
		As above, we separate the cases $\{f(u)<1\}$ and $\{f(u)\geq1\}$. We first rewrite the left-hand side of the above inequality as follows:
		\begin{multline*}
			\varphi_{\delta,\varepsilon}\bigl(z + c\,h(\alpha,\beta,u)\bigr) - \varphi_{\delta,\varepsilon}\bigl(z + c\,h(x,y,u)\bigr) - c\,\bigl(h(\alpha,\beta,u)-h(x,y,u)\bigr)\,\varphi_{\delta,\varepsilon}^{\prime}(z)\\
			= \varphi_{\delta,\varepsilon}\bigl(z + c\,h(x,y,u) + c\,\bigl(h(\alpha,\beta,u)-h(x,y,u)\bigr)\bigr) - \varphi_{\delta,\varepsilon}\bigl(z + c\,h(x,y,u)\bigr)\\
			- c\,\bigl(h(\alpha,\beta,u)-h(x,y,u)\bigr)\bigl(\varphi_{\delta,\varepsilon}^{\prime}(z) - \varphi_{\delta,\varepsilon}^{\prime}\bigl(z + c\,h(x,y,u)\bigr)\bigr)\\ - c\,\bigl(h(\alpha,\beta,u)-h(x,y,u)\bigr)\,\varphi_{\delta,\varepsilon}^{\prime}\bigl(z + c\,h(x,y,u)\bigr).
		\end{multline*}
		Since $\varphi_{\delta,\varepsilon} \in C^2(\R,\R_+)$, we apply Taylor's expansion with integral remainder at order two to the first term of the right-hand side, yielding
		\begin{multline*}
			\varphi_{\delta,\varepsilon}\bigl(z + c\,h(\alpha,\beta,u)\bigr) - \varphi_{\delta,\varepsilon}\bigl(z + c\,h(x,y,u)\bigr) - c\,\bigl(h(\alpha,\beta,u)-h(x,y,u)\bigr)\,\varphi_{\delta,\varepsilon}^{\prime}(z)\\
			= c^2\,\bigl(h(\alpha,\beta,u)-h(x,y,u)\bigr)^2\int_0^1 {(1-r)\,\varphi_{\delta,\varepsilon}^{\prime\prime}\bigl(z + c\,h(x,y,u) + r\,c\,\bigl(h(\alpha,\beta,u)-h(x,y,u)\bigr)\bigr)\ud r}\\
			- c\,\bigl(h(\alpha,\beta,u)-h(x,y,u)\bigr)\bigl(\varphi_{\delta,\varepsilon}^{\prime}(z) - \varphi_{\delta,\varepsilon}^{\prime}\bigl(z + c\,h(x,y,u)\bigr)\bigr).
		\end{multline*}
		The second term of the right-hand side is then coped with by applying Taylor's expansion with integral remainder at order one to $\varphi_{\delta,\varepsilon}^{\prime}$, which gives
		\begin{multline*}
			\varphi_{\delta,\varepsilon}\bigl(z + c\,h(\alpha,\beta,u)\bigr) - \varphi_{\delta,\varepsilon}\bigl(z + c\,h(x,y,u)\bigr) - c\,\bigl(h(\alpha,\beta,u)-h(x,y,u)\bigr)\,\varphi_{\delta,\varepsilon}^{\prime}(z)\\
			= c^2\,\bigl(h(\alpha,\beta,u)-h(x,y,u)\bigr)^2\int_0^1 {(1-r)\,\varphi_{\delta,\varepsilon}^{\prime\prime}\bigl(z + c\,h(x,y,u) + r\,c\,\bigl(h(\alpha,\beta,u)-h(x,y,u)\bigr)\bigr)\ud r}\\
			+ c^2\,h(x,y,u)\,\bigl(h(\alpha,\beta,u)-h(x,y,u)\bigr)\int_0^1 {\varphi_{\delta,\varepsilon}^{\prime\prime}\bigl(z + r\,c\,h(x,y,u)\bigr)\ud r}.
		\end{multline*}
		Denoting the first and second terms of the right-hand side by $\uI$ and $\uII$, respectively, we bound the former by using \eqref{eq:varphi_bis} and observing that $h(\alpha,\beta,u)-h(x,y,u)=h(\alpha,x,u)-h(\beta,y,u)$, as follows:
		\begin{align*}
			\uI &\leq \frac{2\,c^2}{\log\delta}\bigl(h(\alpha,\beta,u)-h(x,y,u)\bigr)^2\int_0^1 {\frac{\ind_{[\varepsilon/\delta, \varepsilon]}\bigl(\bigl|z + c\,h(x,y,u) + r\,c\,\bigl(h(\alpha,\beta,u)-h(x,y,u)\bigr)\bigr|\bigr)}{\bigl|z + c\,h(x,y,u) + r\,c\,\bigl(h(\alpha,\beta,u)-h(x,y,u)\bigr)\bigr|}\,\ud r}\\
			&\leq \frac{2\,\delta\,c^2}{\varepsilon\log\delta}\bigl(h(\alpha,x,u)-h(\beta,y,u)\bigr)^2\int_0^1 {\ind_{[\varepsilon/\delta, \varepsilon]}\bigl(\bigl|z + c\,h(x,y,u) + r\,c\,\bigl(h(\alpha,\beta,u)-h(x,y,u)\bigr)\bigr|\bigr)\,\ud r}\\
			&\leq \frac{4\,\delta\,c^2}{\varepsilon\log\delta}\bigl(h(\alpha,x,u)^2 + h(\beta,y,u)^2\bigr)
			\leq \frac{4\,\delta\,c^2}{\varepsilon\log\delta}\bigl(|x-\alpha| + |\beta-y|\bigr)f(u)^2,
		\end{align*}
		where we have bounded the indicator function by one directly, used $x\leq|x|$ along with Jensen's inequality, and injected $|h(x,y,u)| \leq |x-y|^{1/2} f(u)$ by Assumption~\ref{ass:global_conditions}-(ii). We then bound the second term using again~\eqref{eq:varphi_bis}, $x\leq|x|$, and $h(\alpha,\beta,u)-h(x,y,u)=h(\alpha,x,u)-h(\beta,y,u)$, which gives
		\begin{equation*}
			\uII \leq \frac{2\,c^2}{\log\delta}\bigl|h(x,y,u)\bigr|\bigl|h(\alpha,x,u)-h(\beta,y,u)\bigr|\int_0^1 {\frac{\ind_{[\varepsilon/\delta, \varepsilon]}\bigl(\bigl|z + r\,c\,h(x,y,u)\bigr|\bigr)}{\bigl|z + r\,c\,h(x,y,u)\bigr|}\,\ud r}.
		\end{equation*}
		Using that $2d \tilde{d} \le d^2 + \tilde{d}^2$, the triangle inequality and Jensen's, we obtain
		\begin{align*}
			\uII &\leq \frac{c^2}{\log\delta}\,h(x,y,u)^2\int_0^1 {\frac{\ind_{[\varepsilon/\delta, \varepsilon]}\bigl(\bigl|z + r\,c\,h(x,y,u)\bigr|\bigr)}{\bigl|z + r\,c\,h(x,y,u)\bigr|}\,\ud r}\\
			&\quad+ \frac{2\,\delta\,c^2}{\varepsilon\log\delta}\,\bigl(h(\alpha,x,u)^2 + h(\beta,y,u)^2\bigr)\int_0^1 {\ind_{[\varepsilon/\delta, \varepsilon]}\bigl(\bigl|z + r\,c\,h(x,y,u)\bigr|\bigr)\,\ud r}\\
			&\leq 2\,c^2\biggl(\frac{1}{\log\delta} + \frac{\delta}{\varepsilon\log\delta}\bigl(|x-\alpha| + |y-\beta|\bigr)\biggr)f(u)^2,
		\end{align*}
		where we have again used the fact that $x\mapsto\eta(x,u)$ is non-decreasing to bound the first term (as in the proof of Lemma~\ref{lem:lem_varphi_1}), and injected $|h(x,y,u)| \leq |x-y|^{1/2} f(u)$ as above. 
		
		For $f(u)\geq 1$ , it suffices to make use of $x\leq|x|$, the triangle inequality, \eqref{eq:varphi_bis} and the mean-value theorem to write
		\begin{multline*}
			\varphi_{\delta,\varepsilon}\bigl(z + c\,h(\alpha,\beta,u)\bigr) - \varphi_{\delta,\varepsilon}\bigl(z + c\,h(x,y,u)\bigr) - c\,\bigl(h(\alpha,\beta,u)-h(x,y,u)\bigr)\,\varphi_{\delta,\varepsilon}^{\prime}(z)\\
			\leq \bigl|\varphi_{\delta,\varepsilon}\bigl(z + c\,h(\alpha,\beta,u)\bigr) - \varphi_{\delta,\varepsilon}\bigl(z + c\,h(x,y,u)\bigr)\bigr| + c\,\bigl|h(\alpha,\beta,u)-h(x,y,u)\bigr|\bigl|\varphi_{\delta,\varepsilon}^{\prime}(z)\bigr|\\
			\leq 2\,c\,\bigl|h(\alpha,x,u) - h(\beta,y,u)\bigr|\leq 2\,c\,\bigl(|h(\alpha,x,u)| + |h(\beta,y,u)|\bigr)\\
			\leq 2\,c\,\bigl(|x-\alpha|^{1/2} + |y-\beta|^{1/2}\bigr)f(u),
		\end{multline*}
		where we use again that $h(\alpha,\beta,u)-h(x,y,u)=h(\alpha,x,u)-h(\beta,y,u)$ and Assumption~\ref{ass:global_conditions}-(ii).
	\end{proof}	
	
\end{appendix}

\bibliographystyle{alpha}
\bibliography{biblio_Volterra}

\begin{thebibliography}{BDMKR97}

\bibitem[AJ21]{AJ21}
E.~Abi~Jaber.
\newblock Weak existence and uniqueness for affine stochastic {Volterra}
  equations with {L1}-kernels.
\newblock {\em Bernoulli}, 27(3):1583--1615, 2021.

\bibitem[AJCLP21]{AJCLP21}
E.~Abi~Jaber, C.~Cuchiero, M.~Larsson, and S.~Pulido.
\newblock A weak solution theory for stochastic {Volterra} equations of
  convolution type.
\newblock {\em Ann. Appl. Probab.}, 31(6):2924--2952, 2021.

\bibitem[AJEE19]{AJEE19a}
E.~Abi~Jaber and O.~El~Euch.
\newblock Multifactor approximation of rough volatility models.
\newblock {\em SIAM J. Financial Math.}, 10(2):309--349, 2019.

\bibitem[AJLP19]{AJLP19}
E.~Abi~Jaber, M~Larsson, and S.~Pulido.
\newblock {Affine Volterra processes}.
\newblock {\em Ann. Appl. Probab.}, 29(5):3155--3200, 2019.

\bibitem[Alf05]{Alfonsi05}
A.~Alfonsi.
\newblock On the discretization schemes for the {CIR} (and {Bessel} squared)
  processes.
\newblock {\em Monte Carlo Methods Appl.}, 11(4):355--384, 2005.

\bibitem[Alf23]{Alfonsi23}
A.~Alfonsi.
\newblock Nonnegativity preserving convolution kernels. {Application} to
  {Stochastic} {Volterra} {Equations} in closed convex domains and their
  approximation.
\newblock ArXiv preprint available at \url{https://arxiv.org/abs/2302.07758v1},
  2023.

\bibitem[AN97]{AN97}
E.~Al\`os and D.~Nualart.
\newblock Anticipating stochastic {Volterra} equations.
\newblock {\em Stochastic Process. Appl.}, 72(1):73--95, 1997.

\bibitem[BDMKR97]{BDMKR97}
T.~Bjork, G.~Di~Masi, Y.~Kabanov, and W.~Runggaldier.
\newblock Towards a general theory of bond markets.
\newblock {\em Finance Stoch.}, 1(2):141--174, 1997.

\bibitem[BLP24]{BLP24}
A.~Bondi, G.~Livieri, and S.~Pulido.
\newblock Affine {Volterra} processes with jumps.
\newblock {\em Stochastic Process. Appl.}, 168, 2024.

\bibitem[BM80a]{BM80a}
M.~A. Berger and V.~J. Mizel.
\newblock Volterra equations with {It\^o} integrals--{I}.
\newblock {\em J. Integral Equations Appl.}, 2(3):187--245, 1980.

\bibitem[BM80b]{BM80b}
M.~A. Berger and V.~J. Mizel.
\newblock Volterra equations with {It\^o} integrals--{II}.
\newblock {\em J. Integral Equations Appl.}, 2(4):319--337, 1980.

\bibitem[BPS24]{BPS24}
A.~Bondi, S.~Pulido, and S.~Scotti.
\newblock The rough {Hawkes} {Heston} stochastic volatility model.
\newblock {\em Math. Finance}, n/a(n/a), 2024.

\bibitem[CD01]{CD01}
L.~Coutin and L.~Decreusefond.
\newblock Stochastic {Volterra} equations with singular kernels.
\newblock In A.~B. Cruzeiro and J.-C. Zambrini, editors, {\em Stochastic
  {Analysis} and {Mathematical} {Physics}}, Progress in {Probability}, pages
  39--50. Birkh\"auser, Boston, MA, 2001.

\bibitem[CL95]{ChLe}
J.~Y. Chemin and N.~Lerner.
\newblock Flot de champs de vecteurs non lipschitziens et \'{e}quations de
  {N}avier-{S}tokes.
\newblock {\em J. Differential Equations}, 121(2):314--328, 1995.

\bibitem[Cla87]{Clark87}
D.~S. Clark.
\newblock Short proof of a discrete {{G}r\"onwall} inequality.
\newblock {\em Discrete Appl. Math.}, 16(3):279--281, 1987.

\bibitem[CLP95]{CLP95}
W.~G. Cochran, J.~S. Lee, and J.~Potthoff.
\newblock Stochastic {Volterra} equations with singular kernels.
\newblock {\em Stochastic Process. Appl.}, 56(2):337--349, 1995.

\bibitem[EER19]{EER19}
O.~El~Euch and M.~Rosenbaum.
\newblock The characteristic function of rough {Heston} models.
\newblock {\em Math. Finance}, 29(1):3--38, 2019.

\bibitem[FGS21]{FGS21rate}
C.~Fontana, A.~Gnoatto, and G.~Szulda.
\newblock Multiple yield curve modeling with {CBI} processes.
\newblock {\em Math. Financ. Econ.}, 15(2):579--610, 2021.

\bibitem[FL10]{FL10}
Z.~Fu and Z.~Li.
\newblock Stochastic equations of non-negative processes with jumps.
\newblock {\em Stochastic Process. Appl.}, 120(3):306--330, 2010.

\bibitem[GJR18]{GJR18}
J.~Gatheral, T.~Jaisson, and M.~Rosenbaum.
\newblock Volatility is rough.
\newblock {\em Quant. Finance}, 18(6):933--949, 2018.

\bibitem[GR11]{GR11}
I.~Gy\"ongy and M.~R\'asonyi.
\newblock A note on {Euler} approximations for {SDEs} with {H\"older}
  continuous diffusion coefficients.
\newblock {\em Stochastic Process. Appl.}, 121(10):2189--2200, 2011.

\bibitem[Ham23]{Hamaguchi23}
Y.~Hamaguchi.
\newblock Weak well-posedness of stochastic {Volterra} equations with
  completely monotone kernels and non-degenerate noise, 2023.
\newblock ArXiv preprint available at \url{https://arxiv.org/abs/2310.16030}.

\bibitem[IW89]{IW89}
N.~Ikeda and S.~Watanabe.
\newblock {\em Stochastic {Differential} {Equations} and {Diffusion}
  {Processes}}.
\newblock North-Holland mathematical library. North-Holland,
  Amsterdam--Oxford--New York, second edition, 1989.

\bibitem[JMS17]{JMS17}
Y.~Jiao, C.~Ma, and S.~Scotti.
\newblock {A}lpha-{CIR} model with branching processes in sovereign interest
  rate modeling.
\newblock {\em Finance Stoch.}, 21(3):789--813, 2017.

\bibitem[JMSS19]{JMSS19}
Y.~Jiao, C.~Ma, S.~Scotti, and C.~Sgarra.
\newblock A branching process approach to power markets.
\newblock {\em Energy Economics}, 79:144--156, 2019.

\bibitem[JMSZ21]{JMSZ21}
Y.~Jiao, C.~Ma, S.~Scotti, and C.~Zhou.
\newblock The {Alpha-Heston} stochastic volatility model.
\newblock {\em Math. Finance}, 31(3):943--978, 2021.

\bibitem[KW71]{KW71}
K.~Kawazu and S.~Watanabe.
\newblock Branching processes with immigration and related limit theorems.
\newblock {\em Theory Probab. Appl.}, 16(1):36--54, 1971.

\bibitem[LM11]{LM11}
Z.~Li and L.~Mytnik.
\newblock Strong solutions for stochastic differential equations with jumps.
\newblock {\em Ann. Inst. Henri Poincar\'e Probab. Stat.}, 47(4):1055--1067,
  2011.

\bibitem[LM15]{LM15}
Z.~Li and C.~Ma.
\newblock Asymptotic properties of estimators in a stable
  {Cox--Ingersoll--Ross} model.
\newblock {\em Stochastic Process. Appl.}, 125(8):3196--3233, 2015.

\bibitem[LT19a]{LT19b}
L.~Li and D.~Taguchi.
\newblock On a positivity preserving numerical scheme for jump-extended {CIR}
  process: the alpha-stable case.
\newblock {\em BIT}, 59(3):747--774, 2019.

\bibitem[LT19b]{LT19a}
L.~Li and D.~Taguchi.
\newblock On the {Euler}--{Maruyama} scheme for spectrally one-sided {L\'evy}
  driven {SDEs} with {H\"older} continuous coefficients.
\newblock {\em Statist. Probab. Lett.}, 146:15--26, 2019.

\bibitem[MS15]{MS15}
L.~Mytnik and T.~S. Salisbury.
\newblock Uniqueness for {Volterra}-type stochastic integral equations.
\newblock ArXiv preprint available at \url{http://arxiv.org/abs/1502.05513},
  2015.

\bibitem[PP90]{PP90}
E.~Pardoux and P.~Protter.
\newblock Stochastic {Volterra} equations with anticipating coefficients.
\newblock {\em Ann. Probab.}, 18(4):1635--1655, 1990.

\bibitem[Pro85]{Protter85}
P.~Protter.
\newblock Volterra equations driven by semimartingales.
\newblock {\em Ann. Probab.}, 13(2):519--530, 1985.

\bibitem[PS23]{PS23}
D.~J. Pr\"omel and D.~Scheffels.
\newblock Stochastic {Volterra} equations with {H\"older} diffusion
  coefficients.
\newblock {\em Stochastic Process. Appl.}, 161:291--315, 2023.

\bibitem[RW00]{RW00}
L.~C.~G. Rogers and D.~Williams.
\newblock {\em Diffusions, {Markov} {Processes} and {Martingales}: {Volume} 2:
  {It\^o} {Calculus}}, volume~2 of {\em Cambridge {Mathematical} {Library}}.
\newblock Cambridge University Press, Cambridge, 2 edition, 2000.

\bibitem[Szu21]{phdthesis_szulda}
G.~Szulda.
\newblock {\em Branching Processes and Multiple Term Structure Modeling}.
\newblock PhD thesis, Universit\'e Paris Cit\'e, 2021.

\bibitem[Wan08]{Wang08}
Z.~Wang.
\newblock Existence and uniqueness of solutions to stochastic {Volterra}
  equations with singular kernels and non-{Lipschitz} coefficients.
\newblock {\em Statist. Probab. Lett.}, 78(9):1062--1071, 2008.

\bibitem[Yam78]{Yamada78}
T.~Yamada.
\newblock Sur une construction des solutions d'\'equations diff\'erentielles
  stochastiques dans le cas non-lipschitzien.
\newblock {\em S\'eminaire de probabilit\'es de Strasbourg}, 12:114--131, 1978.

\bibitem[YW71]{YW71}
T.~Yamada and S.~Watanabe.
\newblock On the uniqueness of solutions of stochastic differential equations.
\newblock {\em Kyoto J. Math.}, 11(1):155--167, 1971.

\bibitem[Zha10]{Zhang10}
X.~Zhang.
\newblock Stochastic {Volterra} equations in {Banach} spaces and stochastic
  partial differential equation.
\newblock {\em J. Funct. Anal.}, 258(4):1361--1425, 2010.

\end{thebibliography}

\end{document}